\documentclass[10pt, reqno]{amsart}
\usepackage[a4paper,inner=3cm,outer=3cm,top=3cm,bottom=3cm]{geometry}

\usepackage{a4wide}
\usepackage{amsmath}
\usepackage{amsthm}
\usepackage{amsfonts}
\usepackage{amssymb}
\usepackage[foot]{amsaddr}
\usepackage{graphicx}
\usepackage[english]{babel}
\usepackage{float}
\usepackage{amsthm}
\usepackage{enumerate}
\usepackage{enumitem}
\usepackage{caption}
\usepackage[utf8]{inputenc}
\usepackage{dsfont}
\usepackage{mdframed}
\usepackage{hyperref}
\usepackage{listings}
\usepackage{bbm}
\usepackage{mathtools}

\theoremstyle{plain}
\newtheorem{theorem}{Theorem}[section]
\newtheorem{lemma}[theorem]{Lemma}
\newtheorem{proposition}[theorem]{Proposition}
\newtheorem{corollary}[theorem]{Corollary}
\newtheorem{definition}[theorem]{Definition}

\theoremstyle{definition}

\DeclareMathOperator*{\argmax}{arg\,max}
\DeclareMathOperator*{\argmin}{arg\,min}

\newcommand{\E}{\mathbb{E}}
\newcommand{\prob}{\mathbb{P}}

\newcommand{\sss}[1]{{\scriptscriptstyle #1}}

\numberwithin{equation}{section}

\title[Weighted distances in PA models]{Weighted distances in scale-free preferential attachment models}
\author[J. Jorritsma]{Joost Jorritsma}
\author[J. Komj\'{a}thy]{J\'{u}lia Komj\'{a}thy}
\address{Department of Mathematics and Computer Science\\ Eindhoven University of Technology\\ P.O. Box 513\\ 5600 MB Eindhoven\\ The Netherlands}
\email{j.jorritsma@tue.nl, j.komjathy@tue.nl}
\thanks{Acknowledgements. The work of JJ and JK is partly supported by the Netherlands Organisation for Scientific Research (NWO) through grant NWO 613.009.122.}
\date{July 3, 2019.}

\begin{document}
\begin{abstract}
  We study three preferential attachment models
  where the parameters are such that the asymptotic degree distribution has infinite variance.
  Every edge is equipped with a non-negative i.i.d.\ weight.
  We study the weighted distance between two vertices chosen uniformly at random, the typical weighted distance,
  and the number of edges on this path, the typical hopcount.
  We prove that there are precisely two universality classes of weight distributions, called the explosive and conservative class.
  In the explosive class, we show that the typical weighted distance converges in distribution to the sum of two i.i.d.\ \emph{finite} random variables.
  In the conservative class, we prove that the typical weighted distance tends to infinity, and we give an explicit expression for the main growth term, as well as for the hopcount. Under a mild assumption on the weight distribution the fluctuations around the main term are tight.
\end{abstract}
\maketitle
\section{Introduction}
\subsection{Motivation}
With the rise of social networks like Facebook and Instagram, information spreading in social networks is an actual topic.
What are the main reasons that can make a (fake) message go viral?
A way to model information diffusion mathematically is by representing the topological structure as an (un)directed graph,
where  vertices represent  people, and two vertices are connected by an edge if they are friends in the social network.
Every edge has a non-negative weight attached to it, standing for the time that it takes to transmit a message from one side of the edge to the other.
Other real-world networks with spreading phenomena can be modelled similarly, e.g.\ virus spreading on the internet, epidemics in society.
Many complex networks are intractably large and the underlying graphs are often unknown, let alone the weights.
A simplistic model of such a process is to model the network as an edge-weighted graph  where the edge weights are i.i.d.\ random variables.
This model is called \emph{first passage percolation} (FPP), and was introduced by Hammersley and Welsh for the lattice $\mathbb{Z}^d$ \cite{hammersley1965first}, see also \cite{auffinger201750} and the references therein.
Natural questions for this model are among others:
\begin{itemize}
  \item[(i)] For two vertices $u$ and $v$, what is the transmission time of a message from $u$ to $v$? How does the passage time depend on the size of the graph?
  \item[(ii)] How many edges are on the shortest weighted path from $u$ to $v$? In telecommunication networks, the signal loss increases in the number of edges on a path.
\end{itemize}
As argued above, in complex networks the underlying structure may be intractable.
In order to model these real networks, one can use random graphs that share some topological properties with the original network \cite{newman2003structure}.
We highlight two properties that many real-world networks are believed to share, which are satisfied by the models that we study in the present paper.

 The degree-sequences are scale-free, i.e.,
  the number of connections (degree) per vertex decays as a power law: there exists $\tau>1$
  such that the number of vertices with degree $k$ is proportional to $k^{-\tau}$.
  For example in the WWW, social networks, and protein networks, the power-law exponent $\tau$ is estimated to be in $(2, 3)$, see e.g.\ \cite{albert2002statistical, albert1999internet, newman2003structure}.

Although the networks may contain billions of nodes, the distances, i.e., the minimal number of edges to connect two vertices, are of logarithmic or even doubly logarithmic order of the size of the graph, see \cite{backstrom2012four, montoya2002small, newman2003structure,travers1967small}.
  In these cases, we call the network \emph{small world} or \emph{ultra-small world}, respectively.
  We stress that the small-world property describes distances in \emph{unweighted} graphs, and not in \emph{weighted} graphs as we study in this paper.

Many models have been introduced in the past decades that exhibit these properties,
such as the configuration model (CM) \cite{bender1978asymptotic, bollobas1980probabilistic}, generalized inhomogeneous random graphs \cite{chung2002average}, and the Norros-Reittu model \cite{norros2006conditionally}.
For an extensive discussion and results we refer the reader to \cite{hofstad2016book1} and the references therein.
Arguably one of the most well-known models is the preferential attachment model (PA),
as it gives a possible explanation to the emergence of the power law in the degree sequence \cite{barabasi1999emergence}.
The model became increasingly popular after the paper written by Barab\'asi and Albert \cite{barabasi1999emergence},
although similar models appeared in literature before, see e.g.\ \cite{simon1955class,yule1925ii}.

In the last decade, first passage percolation on random graphs has gained increasing attention, and the process is quite well understood on  \emph{static} graphs, i.e., graphs that do not grow over time.
It has been studied on the Erd\H{o}s-R\'enyi graph in \cite{fppErdos2011}, on
configuration models with finite variance degrees in \cite{bhamidi2010first, Bhamidi2017, bhamidi2014front, dereichOrtgiese2018local_neighbourhoods_cm}, and with infinite variance degrees in  \cite{adriaans2017weighted, baroni2017nonuniversality, Baroni2019, dereichOrtgiese2018local_neighbourhoods_cm}. FPP on spatial models as scale-free percolation,  geometric inhomogeneous random graphs, and hyperbolic random graphs is studied for infinite variance degrees in \cite{sfp2017explosion1, komjathy2018explosion}.
To the extent of our knowledge, no formal results are known for FPP on \emph{dynamically growing} models, such as PA.

In this paper we study three (non-spatial) PA models that are among the most commonly studied in literature.
We introduce them informally.
The construction of a graph is initialized with a graph $\mathrm{PA}_1$ and arrivals of vertices happen deterministically at times $t\in\{2, 3, ...\}$. The graph on $t$ vertices is denoted by $\mathrm{PA}_t$.
Arriving vertices favour connecting to vertices with high degree such that the asymptotic degree distribution follows a power law with exponent $\tau\in(2,3)$. We call $\tau$ the power-law exponent.
We call the number of connections that a new vertex establishes upon arrival the \emph{outdegree}. We study three variants of this model: one with fixed outdegree (FPA), and two with variable outdegree (VPA and GVPA), see Definitions \ref{def:pa_m_delta} and \ref{def:pa_gamma} below.
Once an edge is created, we equip the edge with an edge-weight, an i.i.d.\ copy of a non-negative random variable $L$.
Three different distance \emph{metrics} are of specific interest:
the \emph{typical weighted distance} $d_L^{(t)}(u,v)$, i.e., the total weight on the least weighted path between two vertices chosen uniformly at random, independently of each other in $\mathrm{PA}_t$, called \emph{typical vertices};
the \emph{typical hopcount} $d_H^{(t)}(u,v)$, i.e., the number of edges on the least weighted path between typical vertices;
  and the \emph{typical graph distance} $d_G^{(t)}(u,v)$, i.e., the number of edges on the shortest path between  typical vertices in the unweighted graph.

\subsection{Our contribution}\label{sec:contribution}
We prove that there are exactly two universality classes of weight distributions for the three edge-weighted models with power-law exponent $\tau\in(2,3)$. The universality classes are determined by a computable characteristic of the weight distribution $L$. We present the characteristic and an informal version of our main result here, precise results can be found in Theorems \ref{thm:weighted_distance_nonexplosive} and \ref{thm:weighted_distance_explosive} below.
For a random variable $L$, we define its cumulative distribution function as $F_L(x):=\prob(L\leq x)$,
and its generalized inverse by $F^{(-1)}(y):=\inf_x\{x\in\mathbb{R}: F(x)\geq y\}$.

\begin{definition}[Explosion characteristic {$I(L)$}]\label{def:explosion_char}
  Let $L$ be a non-negative random variable with distribution function $F_L$. We define the explosion characteristic $I(L)$ as
  \begin{equation}
    I(L) := \sum_{k=1}^\infty F_L^{(-1)}\left(\mathrm{e}^{-\mathrm{e}^k}\right).
    \label{eq:il_sum_of_inverse_distribution}
  \end{equation}
  We call $\{L: I(L)=\infty\}$ the conservative class, and $\{L:I(L)<\infty\}$ the explosive class.
\end{definition}
\noindent The term \emph{explosion} originates from the study of age-dependent branching processes, see e.g.\ \cite{amini2013explosions, athreya2004branching, harris2002theory}.
In these branching processes we say that explosion happens if infinitely many individuals are born within finite time.
The relation to explosion in trees comes from the fact that the neighbourhood of a typical vertex in $\mathrm{PA}_t$ converges in distribution to a random tree, the local weak limit. Local weak convergence is shown for the three models in \cite{berger2014asymptotic, dereich2013random, garavaglia2019preferential}.
It is interesting in its own right to study the edge-weighted version of the local weak limit tree. We prove that infinitely many vertices are within finite weighted distance from the root in the local weak limit if the weight distribution is in the explosive class. This fact is then used to show convergence in distribution for the typical weighted distance in $\mathrm{PA}_t$ if $I(L)<\infty$, the first part of our following main result.
\medskip
\begin{mdframed}
  \begin{theorem}[Meta theorem] \label{thm:meta-theorem}
    Consider $\mathrm{PA}_t$ with power-law exponent $\tau\in(2,3)$.
    If $I(L)<\infty$, then
    the typical weighted distance converges in distribution to an almost surely finite random variable. If $I(L)=\infty$, then
    \begin{equation*}
    d_L^{(t)}(u, v) \approx 2Q_t + o_{\prob}(Q_t),
  \end{equation*}
    where
    \begin{equation}
      K^\ast_t:=\left\lfloor\left. 2\log\log(t) \middle / |\log(\tau-2)| \right. \right\rfloor,
      \quad
      Q_t := \sum_{k\leq K^\ast_t} F_L^{(-1)}\left(\exp\big(-(\tau-2)^{-k/2}\big)\right). \label{eq:k_ast_qt}
    \end{equation}
    Under a mild extra condition on $L$ the error term $o_{\prob}(Q_t)$ is tight.
  \end{theorem}
\end{mdframed}
\medskip
\noindent
This meta theorem is formalized in Section \ref{sec:main-results} after the model definitions. There, we describe the limiting random variables if $I(L)<\infty$ and state the computable, mild condition on $L$ that yields tight error terms.
As a side result of the second part of Theorem \ref{thm:meta-theorem}, when $I(L)=\infty$, we show that if the weights are of the form $1+X$, $I(X)<\infty$, the typical weighted distance and typical hopcount are both tight around $4\log\log(t)/|\log(\tau-2)|$ for the models FPA and VPA.
This indicates that the addition of an \emph{excess} edge-weight beyond one does not affect the topology of the shortest paths drastically.
Constant weights are a special case of these. So, our result extends results from \cite{caravenna2016diameter, dereich2012typical,dommers2010diameters, monch2013distances}, by showing that the fluctuations of the typical graph distance around $4\log\log(t)/|\log(\tau-2)|$ are tight.

\subsection*{Organisation}
The next section formally introduces the models and the necessary concepts to describe the limiting random variables for the explosive case.
Afterwards, we state our main results, and discuss them by formulating some open problems and recalling relevant results from literature.
In Section \ref{sec:upperbound}, we prove  upper bounds for the weighted distance in the finite graphs by constructing a path, and show that the local weak limit tree is explosive if and only if the edge-weight distribution $L$ is a member of the explosive universality class.
Then, in Section \ref{sec:lowerbound}, we prove the corresponding lower bounds for both the conservative as the explosive regime.
In Section \ref{sec:hopcount} we prove a theorem on the hopcount.
Lastly, in Section \ref{sec:cons-lwl} we extend the results for conservative distributions on finite graphs to the local weak limit.
\subsubsection*{Notation}
For $\min\{m, n\}$ and $\max\{m, n\}$ we write respectively $m\wedge n$ and $m\vee n$.
Furthermore, $\lceil x\rceil:=\min\{y\in\mathbb{Z}, y\geq x\}$ and $\lfloor x \rfloor:=\max\{y\in\mathbb{Z}, y\leq x\}$. For $n\in\mathbb{N}$, the set $\{1, 2, ..., n\}$
is denoted by $[n]$.
If in a graph there is an edge incident to both $u$ and $v$, we write $u\leftrightarrow v$,
whereas for a set of vertices $S$, we write $u\leftrightarrow S$ if there is a vertex
$v$ in $S$ such that $u\leftrightarrow v$.
A sequence of events $(\mathcal{E}_n)_{n\geq1}$ holds with high probability (whp) if $\lim_{n\rightarrow\infty}\prob(\mathcal{E}_n)=1$. We abbreviate with probability by w/p.
If the weight random variable $L$ is indexed, then we assume that the random variables with different indices are i.i.d.
\section{Models and main results}
\subsection{Models}\label{sec:models}
We introduce three models, one where the number of outgoing edges is fixed, whereas in the second model the number of outgoing edges is
a random variable and the third model allows for more general connection probabilities than the second model.
The first model we call fixed preferential attachment (FPA).
This model appeared formally for less general cases in \cite{berger2005spread, bollobas2004diameterpa}.
For a complete introduction, we refer to \cite[Chapter 8]{hofstad2016book1}.
At time $t$ we sequentially add $m$ outgoing edges to the arriving vertex labelled $t$.
After the $j$-th edge has been formed, the defining connection probabilities are updated.
Denote by $D^{\leftarrow}_{(t, j)}(v)$ the indegree, the number of incoming edges, of a vertex $v$
  right after the $(tm+j)$-th edge is added to the graph.
Similarly, we write $\mathrm{FPA}_{(t, j)}$ for the constructed graph
  right after the moment it contains exactly $(tm+j)$ edges.
  Let $\{t\overset{j}{\rightarrow}v\}$ be
    the event that the $j$-th edge of vertex $t\in\mathbb{N}$ is attached to $v\in[t-1]$.
\begin{definition}[$\mathrm{FPA}(m,\delta)$]\label{def:pa_m_delta}
  Fix $m\in\mathbb{N}, \delta\in(-m,\infty)$.
  Let $\mathrm{FPA}_{1}(m, \delta)$ be the graph with a single vertex without any edges.
  The model $\mathrm{FPA}(m,\delta)$ is defined by the following sequence of   conditional connection probabilities
  \begin{equation}\label{eq:pa_mdelta_connection}
    \prob\big(\{t\overset{j}{\rightarrow}v\}\mid \mathrm{FPA}_{(t, j)}\big) = \frac{D^{\leftarrow}_{(t, j-1)}(v) + m\left(1 + \delta/m\right)}{Z_{t,j}}, \qquad v\in[t-1],
  \end{equation}
  where $Z_{t,j}$ is a normalizing constant.
  The power-law exponent of the model is
  \begin{equation}\label{eq:tau_m_delta}
    \tau_{m,\delta} := 3+ \delta/m.
  \end{equation}
\end{definition}
\noindent
One can verify that $Z_{t,j}=(t-2)(\delta+2m)+ j-1 + m+\delta$.
There are many variants of FPA \cite{hofstad2016book1}.
Our definition does not allow for self-loops, but allows for multi-edges.
Some variants behave qualitatively similarly and our results extend to the models in \cite{hofstad2016book1}.
The numerator in the connection probabilities in \eqref{eq:pa_mdelta_connection} is equivalent to \cite[Formula (8.2.1)]{hofstad2016book1}, where the numerator is a function of the \emph{total degree} rather than the \emph{indegree}. The two formulas coincide, because the outdegree is equal to $m$.
The asymptotic degree distribution in FPA decays as a power law with exponent $\tau_{m,\delta}$ in \eqref{eq:tau_m_delta}, see \cite{hofstad2016book1}, so that for FPA $\tau\in(2,3)$ when $\delta\in(-m,0)$ as in various real-world networks \cite{newman2003structure}.
The constraint $\delta>-m$ ensures well-defined probabilities in \eqref{eq:pa_mdelta_connection}.

In FPA the total number of edges in the graph is deterministic, making some explicit calculations easier.
On the contrary, the events $\{t \rightarrow v_1\}$ and $\{t\rightarrow v_2\}$ are negatively correlated for vertices $v_1\neq v_2$,
 yielding more involved computations. The next models that we introduce behave to some extent as the opposite of FPA,
as the edges are conditionally independent, leading to a random outdegree of vertices. They were introduced by Dereich and M\"orters \cite{dereich2009random},
where they call the graph \emph{preferential attachment with conditionally independent edges}, although in \cite{hofstad2016book1} it is called \emph{Bernoulli preferential attachment}.
Let $D_t^{\leftarrow}(v)$ be the indegree of the vertex $v$ right before time $t$.
\noindent
\begin{definition}[{$\mathrm{VPA}(f)$, $\mathrm{GVPA}(f)$}]\label{def:pa_gamma}
  Let $f:\mathbb{N}\rightarrow (0,\infty)$ be a concave function satisfying $f(0)\leq 1$ and $f(1) - f(0)<1$.
  We call $f$ the \emph{attachment rule}.
  Let $\mathrm{GVPA}_1(f)$ be the graph with a single vertex without any edges.
  Conditionally on $\mathrm{GVPA}_{t-1}(f)$, vertex $t$ connects to $v\in[t-1]$ w/p
  \begin{equation*}
  f\left(D^{\leftarrow}_{t}(v)\right)/t,
\end{equation*}
  independently of the other existing vertices.
  Important parameters of the model are
  \begin{equation}
    \gamma_f := \lim_{t\rightarrow\infty}f(t)/t, \qquad \tau_{f}:= 1 + 1/\gamma,\label{eq:gamma_f}
  \end{equation}
  which are well-defined by the concavity of $f$. We call $\tau_f$ the power-law exponent.
  For general $f$, we call the model generalized variable preferential attachment (GVPA). For affine $f$, we call the model variable preferential attachment (VPA).
\end{definition}
\noindent For results on these models, such as the size of the giant component, and the asymptotic degree distribution, we refer the reader to \cite{dereich2012typical, dereich2009random, dereich2011random, dereich2013random}.
In the model VPA, calculations become explicit and precise results can be derived.
The asymptotic degree distribution decays as a power law with
exponent $\tau_\gamma$ in \eqref{eq:gamma_f}, see \cite{dereich2009random}.
We assume that $\gamma_f\in(1/2,1)$, so that $\tau_f\in(2,3)$.

The three models, FPA, VPA, and GVPA, behave qualitatively similar in terms of their degree distribution and \emph{typical graph distance} when $\tau\in(2,3)$ \cite{dereich2012typical}.
This motivates to refer to the models by their power-law exponent $\tau$, and to call them PA collectively. We distinguish them only when different proofs are required, or when referred to different results from literature.

In the present paper, we look at \emph{typical least weighted paths}, that is, we assume that every edge in $\mathrm{PA}_t$ is equipped with an i.i.d.\ weight, and we are interested
in the sum of the weights on the least weighted path between two vertices, and the number of edges on this path, called \emph{hopcount}.
\begin{definition}[Distances in graphs]\label{def:graph_distances_1}
    Consider the graph $\mathcal{G}=\left(V,E\right)$
    and assume every edge $e\in E$ is equipped with a weight $L_e$.
    For a path $\pi$, we define its length as $\|\pi\|:=\sum_{e\in\pi}1$
    and $L$-length as $\|\pi\|_L:=\sum_{e\in\pi}L_e$.
    For $u, v\in V$, let $\Omega_{u,v}:=\{\pi: \pi \text{ is a path from $u$ to $v$ in $\mathcal{G}$}\}$.
    We define the distance, $L$-distance (also called weighted distance), and $H$-distance (hopcount)
    between $u$ and $v$ in the graph $\mathcal{G}$ as
    \begin{align*}
      d_G(u,v) := \min_{\pi\in\Omega_{u,v}} \|\pi\|,
      \qquad
      d_L(u,v) := \min_{\pi\in\Omega_{u,v}} \|\pi\|_L,
      \qquad
      d_H(u, v):= \Big\|\argmin_{\pi\in\Omega_{u,v}} \|\pi\|_L\Big\|,
    \end{align*}
    respectively.
    If $\Omega_{u,v}=\varnothing$, the above distance-metrics are defined as $\infty$. If there are several paths $\pi_1,\dots,\pi_k$ minimizing the $L$-distance, the hopcount is defined as $\min_{i\leq k}\|\pi_i\|$.
    For a letter $\square\in\{G,L,H\}$, the typical $\square$-distance of a graph $\mathcal{G}$ is defined as the $\square$-distance
    between two typical vertices.
    For $q\in V$ and a set $A\subseteq V$, we generalize distances and define the $\square$-diameter of $A$ by
    \[
      d_{\square}(q,A) = \min_{w\in A}d_{\square}(q, w), \qquad
      \mathrm{diam}_{\square}(A)=\max_{x,y\in A}d_{\square}(x,y).
    \]
    For a vertex $q$, its $\square$-neighbourhood with radius $r>0$ and its boundary are defined as
    \begin{equation*}
      \mathcal{B}_{\square}(q,r) := \{w: d_{\square}(q, w)\leq r\}, \qquad
      \partial\mathcal{B}_G(q,r) := \{w: d_G(q, w)= \lfloor r\rfloor\}.
    \end{equation*}
    We write $\widetilde{\mathcal{B}}_{\square}(q, r)$ for the induced subgraph of $\mathcal{G}$
    on the vertex set $\mathcal{B}_{\square}(q,r)$, with edges $(u,v)$ from $\mathcal{G}$ if both $u$ and $v$ are in $\mathcal{B}_{\square}(q,r)$. When the considered graph $\mathcal{G}$ is not directly clear from the context, we add a superscript $(\mathcal{G})$ to the various notions.
    If $\mathcal{G}=\mathrm{PA}_t$, we abbreviate $(\mathrm{PA}_t)$ by $(t)$ in the superscript.
\end{definition}
An alternative way to look at an edge-weighted graph is to view the weights as \emph{passage times}, i.e., the time that it takes to send a message from one side of the edge to the other. The notions of time and weight are used interchangeably. We stress that the passage time of a single edge is not related to the time $t$ in the construction of the graph $\mathrm{PA}_t$.

We introduce the concepts of explosion time and local weak limit
to describe the limiting random variables for the typical weighted distance in the explosive class in Theorem \ref{thm:weighted_distance_explosive} below.
\begin{definition}[Explosive graph]\label{def:explosive_graph}
    Let $\mathcal{G}=(V,E)$ be a weighted graph  that is locally finite. For the time to reach graph distance $k$ and the time to its $n$-th closest vertex in $L$-distance from a vertex $q$, we write
    \[
    \beta^{(\mathcal{G})}_k(q)\,:=\,d_L\left(q, \partial \mathcal{B}_G(q, k)\right), \qquad
    \sigma^{(\mathcal{G})}_n(q) := \inf\left\{r: |\mathcal{B}_L(q, r)|\geq n\right\}.
    \label{eq:sigma_n}
    \]
  If $|V|=\infty$, we define the explosion time of $q$  as
  $
  \beta_\infty^{(\mathcal{G})}(q) := \lim_{k\rightarrow\infty}\beta^{(\mathcal{G})}_k(q).
  $
  If there is a $q\in V$ with finite explosion time, then we call $\mathcal{G}$ explosive.
  For $\mathcal{G}\equiv \mathrm{PA}_t$, we write $\beta_k^{(t)}$ and $\sigma_n^{(t)}$
  if the $q$ in $\mathrm{PA}_t$ is a typical vertex.
  If $\mathcal{G}$ is a tree rooted in $\circledcirc$, we abbreviate $\beta_k^{(\mathcal{G})}:=\beta_k^{(\mathcal{G})}(\circledcirc)$.
\end{definition}
\noindent The local weak limit of graphs can be used to describe the neighbourhood of a typical vertex. For an introduction we refer to \cite[Chapter 2]{hofstad2017stochasticprocesses} and its references.
Let $\mathbb{G}_\star$ be the space of all (possibly infinite) rooted graphs.
\begin{definition}[Local weak limit in probability]
  Let $(\mathcal{G}_t)_{t\geq 0}$ be a sequence of finite random rooted graphs, and
  let $(\mathcal{G},q)$ be a rooted random graph following law $\mu$. The sequence $(\mathcal{G}_t)_{t\geq 0}$ converges in probability in the local weak convergence sense to $(\mathcal{G},q)$, when
  \[
  \mathbb{E}_t[h(\mathcal{G}_t, q_t)]\overset{\prob}{\longrightarrow}\mathbb{E}[h(\mathcal{G}, q)],
  \]
  for every bounded and continuous function $h:\mathbb{G}_\star\rightarrow\mathbb{R}$, where the expectation on the rhs is w.r.t. $(\mathcal{G},q)$ having law $\mu$, while the expectation on the lhs is w.r.t. the typical vertex $q_t$ only.
\end{definition}
\noindent Berger \emph{et al.} \cite{berger2014asymptotic} identify the local weak limit of FPA.
They give an explicit construction of the limit that they call the \emph{P\'olya-point graph} ($\mathrm{PPG}$),
an infinite rooted tree derived from a multi-type branching process.
While the construction of FPA in \cite{berger2014asymptotic} is slightly different from Definition \ref{def:pa_m_delta}, it can be related to our model for $\delta\geq0$.
 In \cite[Chapter 4]{garavaglia2019preferential}, it is shown that the result remains valid for a wider class of models, in particular when $\delta<0$.

Turning to the local weak limit of GVPA, Dereich and M\"orters \cite{dereich2013random} introduce a similar concept for the  GVPA-model, the \emph{idealized neighbourhood tree} (INT). While local weak convergence is only stated  briefly before \cite[Theorem 1.8]{dereich2013random}, they construct a coupling similar to the PPG.
For our proofs it is not important how the local weak limits can be constructed, only that they exist.
In fact, we consider the $\mathrm{PPG}$ and $\mathrm{INT}$ as a \emph{black box} and yet obtain results.
If a statement holds for both models, we refer to the INT or PPG as LWL (local weak limit).
We write $\mathrm{LWL}_k$ for the tree restricted to vertices that have graph distance at most $k$ from the root.
We call the vertices that are at graph distance exactly $k$ away from the root the $k$-th generation of the LWL.
We state the combined result on local weak convergence for the reader's convenience.
We call two rooted graphs $(\mathcal{G}, x)$, $(\mathcal{G}', x')$ rooted isomorphic, and write $(\mathcal{G},x)\simeq(\mathcal{G}',x')$, if there exists an isomorphism
from $\mathcal{G}$ to $\mathcal{G}'$ that maps $x$ to $x'$.
Recall the graph-neighbourhood $\widetilde{\mathcal{B}}_G$ from Definition \ref{def:graph_distances_1}.
\begin{proposition}[Local weak convergence
  {\cite[Theorem 2.2, Proposition 3.6]{berger2014asymptotic}},
  {\cite[Section 5,6]{dereich2013random}}]\label{prop:local_weak_limit}
  The local weak limits of PA are the P\'olya-point graph for $\mathrm{FPA}(m,\delta)$, and the idealized neighbourhood tree
  for $\mathrm{GVPA}(f)$.
  Moreover, let $q$ be a typical vertex, then for all $\delta_{\ref{prop:local_weak_limit}}>0$ there exists a function $\kappa_{\delta_{\ref{prop:local_weak_limit}}}(t)$ that tends to infinity with $t$, such that
  $\mathcal{B}^{(t)}_G(q, \kappa_{\delta_{\ref{prop:local_weak_limit}}}(t))$ and $\mathrm{LWL}_{\kappa_{\delta_{\ref{prop:local_weak_limit}}(t)}}$ can be coupled, such that,
  denoting by $\circledcirc$ the root of the $\mathrm{LWL}$,
  \begin{equation}
  \prob\left(\widetilde{\mathcal{B}}^{(t)}_G\left(q, \kappa_{\delta_{\ref{prop:local_weak_limit}}}(t)\right)\simeq \mathrm{LWL}_{\kappa_{\delta_{\ref{prop:local_weak_limit}}(t)}}(\circledcirc)\right) \geq 1-\delta_{\ref{prop:local_weak_limit}},
  \label{eq:local_weak_limit}
  \end{equation}
\end{proposition}
\subsection{Main results}\label{sec:main-results}
Recall the explosion characteristic $I(L)$ from \eqref{eq:il_sum_of_inverse_distribution}.
We start with the results on the typical weighted distance in $\mathrm{PA}_t$, where the edge-weight distribution satisfies $I(L)=\infty$.  In this case the typical weighted distance tends to infinity as the graph size tends to infinity.  We determine the first order of growth and the number of edges used on this path. For FPA and VPA, we strengthen our results by showing that the fluctuations around the first order term are tight under a mild condition on $L$.
Recall $K^\ast_t$ and $Q_t$ from \eqref{eq:k_ast_qt}.
\begin{theorem}[Weighted distance, conservative case]\label{thm:weighted_distance_nonexplosive}
  Consider PA with power-law exponent $\tau\in(2, 3)$, i.i.d.\ weights on the edges with distribution $F_L$
  satisfying $I(L)=\infty$.
  Let $u, v$ be two typical vertices.
  Then, for the typical weighted  distance in $\mathrm{PA}_t$,
  \begin{equation}
    \left. d_L^{(t)}(u, v) \middle / 2Q_t \right.
    \overset{\prob}{\longrightarrow} 1, \qquad \text{ as }t\rightarrow\infty.\label{eq:weighted-distance-tight}
  \end{equation}
  Moreover, for the models FPA and VPA from Definition \ref{def:pa_m_delta} and \ref{def:pa_gamma},
  if $I(L)=\infty$ and $F_L$ satisfies
  \begin{equation}\label{eq:tightness_cond}
    \sum_{k=1}^\infty \frac{1}{k}\left(F_L^{(-1)}\left(\mathrm{e}^{-\mathrm{e}^k}\right) - \sup\{x: F_L(x)=0\}\right) < \infty,
  \end{equation}
  then
  \begin{equation*}
    \left( d_L^{(t)}(u,v) - 2Q_t\right)_{t\geq 1}
    % \label{eq:weighted_distance_nonexplosive_thm}
  \end{equation*}
  forms a tight sequence of random variables, i.e., the fluctuations are of order $O(1)$ whp.
\end{theorem}
\noindent We believe that \eqref{eq:tightness_cond} is only a technical condition.
Only distributions $L$ that are extremely \emph{flat} around the origin (triple exponentially) violate it.
An artificial example of such a distribution is if $F_L$ in the neighbourhood
of 0 satisfies
\[
F_L(x) = \exp\{-\exp\{\mathrm{e}^{x^{-\beta}}\}\}
\]
for some $\beta\geq1$.
If $F_L$ satisfies this equality for some $\beta\in(0, 1)$, then condition \eqref{eq:tightness_cond} is satisfied.

We proceed to the typical hopcount for a class of conservative weight distributions.
\begin{theorem}[Hopcount, weights bounded away from zero]\label{thm:hopcount_nonexplosive}
  Consider PA with power-law exponent $\tau\in(2, 3)$, i.i.d.\ weights on the edges with distribution $F_L$ whose support is bounded away from zero, i.e., ${a:=\sup\{x: F_L(x)=0\}>0}$\footnote{This implies that $I(L)=\infty$.}.
  Let $u, v$ be two typical vertices.
  Then, for the typical hopcount in $\mathrm{PA}_t$
  \begin{equation}
    \left. d_H^{(t)}(u, v) \middle / 2K^\ast_t \right.
    \overset{\prob}{\longrightarrow} 1, \qquad \text{ as }t\rightarrow\infty.
    \label{eq:thm_hopcount_nontight}
  \end{equation}
  Moreover, for the models FPA and VPA from Definition \ref{def:pa_m_delta} and \ref{def:pa_gamma}, if $I(L-a)<\infty$, then
  \begin{equation}
    \left(d_H^{(t)}(u, v) -  2K^\ast_t \right)_{t\geq 1}
    \label{eq:thm_hopcount_tight}
  \end{equation}
    forms a tight sequence of random variables.
\end{theorem}
\noindent Setting the weights $L\equiv 1$ in Theorem \ref{thm:hopcount_nonexplosive} immediately implies  the following corollary,  extending results in \cite{caravenna2016diameter, dereich2012typical} on the typical graph distance in FPA up to tight error terms, and confirming the tight error terms for VPA from \cite{monch2013distances}.
\begin{corollary}[Tight graph distances]
  Consider FPA or VPA with power-law exponent $\tau\in(2, 3)$.
  Let $u, v$ be two typical vertices.
  Then, for the typical graph distance in $\mathrm{PA}_t$
  \[
  \left(d_G^{(t)}(u, v) -  2K^\ast_t \right)_{t\geq 1}
  \]
  forms a tight sequence of random variables.
\end{corollary}

Before we move on to the results on finite graphs for weight distributions satisfying $I(L)<\infty$,
we discuss first passage percolation on the LWL. The following theorems show that the LWL is explosive if and only if $I(L)<\infty$. We start with  conservative edge weights.
Afterwards we prove explosiveness of the LWL for explosive edge weights. This  is then used to state the last theorem on the weighted distances in $\mathrm{PA}_t$ for explosive edge weights.
\begin{theorem}[FPP on the LWL, conservative case]\label{theorem:conservative_lwl}
  Consider a PPG or INT rooted in $\circledcirc$ with power-law exponent $\tau\in(2,3)$ with
  i.i.d.\ weights on the edges with distribution $F_L$ satisfying $I(L)=\infty$.
  If $F_L$ satisfies \eqref{eq:tightness_cond}, then, for FPA and VPA,
  \begin{equation}
  \bigg( \beta_k^{(\mathrm{LWL})}(\circledcirc) - \sum_{i=1}^k F_L^{(-1)}\left(\exp\big(-(\tau-2)^{-i/2}\big)\right) \bigg)_{k\geq 1}\label{eq:lwl-tightness}
  \end{equation}
  is a tight sequence of random variables.
  Regardless of \eqref{eq:tightness_cond}, for FPA, VPA, and GVPA, as $k$ tends to infinity,
  \begin{equation}
  \left. \beta_k^{(\mathrm{LWL})}(\circledcirc)\middle/\sum_{i=1}^k F_L^{(-1)}\left(\exp\big(-(\tau-2)^{-k/2}\big)\right) \right. \overset{a.s.}\longrightarrow 1.
  \label{eq:lwl-no-tightness}
\end{equation}
\end{theorem}
\noindent
Observe the similarities between Theorem \ref{thm:weighted_distance_nonexplosive} and Theorem \ref{theorem:conservative_lwl}. In fact, the proof of Theorem \ref{theorem:conservative_lwl} heavily relies on couplings between the LWL and $\mathrm{PA}_t$, which are possible by Proposition \ref{prop:local_weak_limit}. These couplings allow for the intermediate lemmas and propositions to consider whichever object, i.e., $\mathrm{PA}_t$ vs.\ LWL, is more suitable and lead to the similarities between the two theorems.

We now proceed to the universality class of weight distributions satisfying $I(L)<\infty$.
This holds for most \emph{well-known} distributions with support starting at 0, e.g.\ the exponential distribution.
After stating that for these weight distributions the LWL is explosive, we proceed with a theorem on the typical weighted distance in finite graphs. Recall Definition \ref{def:explosive_graph} of an explosive graph.
\begin{theorem}[FPP on the LWL, explosive case]\label{theorem:explosion_time_ppg}
  Consider a PPG or INT rooted in $\circledcirc$ with power-law exponent $\tau\in(2,3)$ with
  i.i.d.\ weights on the edges with distribution $F_L$ satisfying $I(L)<\infty$.
  Then the explosion time of the LWL is an almost surely finite random variable, i.e.,
  \begin{equation}
  \prob\left(\beta_{\infty}^{\mathrm{LWL}} < \infty \right) = 1.
  \nonumber%\label{eq:explosion_time_ppg}
  \end{equation}
\end{theorem}
\begin{theorem}[Weighted distance, explosive case]\label{thm:weighted_distance_explosive}
  Consider PA with power-law exponent $\tau\in(2, 3)$, i.i.d.\ weights on the edges with distribution $F_L$
  satisfying $I(L)<\infty$.
  Let $u, v$ be two typical vertices.
  Then, for the typical weighted distance in $\mathrm{PA}_t$
  \begin{equation}
    d_L^{(t)}(u,v) \overset{d}{\longrightarrow} \beta_{\infty}^{(1)} + \beta_{\infty}^{(2)}, \qquad \text{ as }t\rightarrow\infty,
    \nonumber%\label{eq:weighted_distance_explosive}
  \end{equation}
  where $\beta_{\infty}^{(1)}$ and $\beta_{\infty}^{(2)}$ are two i.i.d.\ copies of the explosion time of the LWL.
\end{theorem}
\noindent
It is remarkable that the limiting random variable does \emph{not} depend on $t$ and thus the graph distance is of much larger order than the weighted distance.
The underlying intuition is that in the graph-neighbourhoods of $u$ and $v$ there is a vertex with sufficiently high degree. The weighted distances to these vertices converge in distribution to $\beta_\infty^{(u)}$ and $\beta_\infty^{(v)}$.
There are many paths connecting these high degree vertices, where the number of edges on these paths is \emph{similar} to the graph distance, allowing to bound its total weight from above and show that it tends to zero.

\subsection{Discussion and open problems}
This paper obtains the first results on weighted distances in preferential attachment models.
The same universality classes of weight distributions from Definition \ref{def:explosion_char}  appear for static models as the Configuration Model (CM)  \cite{adriaans2017weighted, baroni2017nonuniversality, Baroni2019}, and for the spatial models scale-free percolation (SFP), geometric inhomogeneous random graphs (GIRG), hyperbolic random graphs (HRG) \cite{sfp2017explosion1, komjathy2018explosion}, when $\tau\in(2,3)$. Although parts of our proof techniques are similar to techniques used for FPP in SFP, GIRG, HRG, and CM \cite{adriaans2017weighted, baroni2017nonuniversality,Baroni2019,komjathy2018explosion}, we claim that
the \emph{sprinkling argument} demonstrated below for the upper bound on the weighted distance is more general than the delicate degree-dependent or weight-dependent percolation arguments that were applied to these \emph{static} models. Our technique developed for a \emph{dynamically growing} model can be adapted to obtain similar results on static models for weight distributions in the conservative class. In particular, the improved recursion on $(s_k)_{k\geq 0}$ in \eqref{eq:sk} below can be used to prove tightness of typical weighted distances in CM for $\tau\in(2,3)$ around the main term under condition \eqref{eq:tightness_cond}, proving part of \cite[Problem 2.10]{adriaans2017weighted}.

For a fixed $\tau\in(2,3)$, the main difference between CM and PA is that all distances in PA are roughly twice the distance in CM. For CM, the main term of the \emph{graph} distance is $2\log\log(t)/|\log(\tau-2)|$ \cite{hofstad2005distancescminfinitevariance}, compared to $4\log\log(t)/|\log(\tau-2)|$ in PA.
Combining the results on weighted distances in CM \cite{adriaans2017weighted} with our results, after a variable transformation on the sum in $Q_t$ in \eqref{eq:k_ast_qt}, the factor two extends to \emph{weighted} distances, i.e.,
\[
\left. d_L^{\mathrm{CM}_t(\tau)}(u, v) \middle / d_L^{\mathrm{PA}_t(\tau)}(u, v)\longrightarrow 2\right.,\qquad\text{as $t\rightarrow\infty$.}
\]
It is commonly believed this is due to the difference in construction: in CM high degree vertices are often directly connected via an edge, while in PA we need two edges to connect two high degree vertices.
This factor two is studied in \cite{dereich2017distances_pa_critical}, where the authors show how this factor two vanishes
in GVPA($f$) with power-law exponent $\tau=3$ and a function $f$ that has logarithmic corrections.

For some other dynamically growing models results on graph distances are not yet known, although degree distributions and clustering coefficients have been studied. Examples are preferential attachment with edge steps \cite{alves2019agglomeration, alves2017preferential}, and the age-dependent random connection model \cite{gracar2018age}. For spatial preferential attachment as introduced in \cite{jacob2013spatial, jacob2015spatial}, an upper bound on the graph distance is established in \cite{hirsch2018distances_spatial_pa}. In the proof for the upper bound, the authors show that two vertices with high degree are connected via two edges, similarly to the non-spatial models.
 To the extent of the authors' knowledge, no matching lower bound is known.
It would be of interest to see if the typical graph distance on these models grows at order $\Theta(\log\log(t))$ as well, and if our technique to determine the typical weighted distance translates to these models.

 Little is known for graph distances in FPA and GVPA when $\tau>3$. For FPA it is shown in \cite{dommers2010diameters} that the diameter of the graph and the typical graph distance are of order $\Theta(\log(t))$, but the precise main order of growth remains unknown.
 This is in sharp contrast with CM, where the graph distance grows as $\Theta(\log(t))$ and the precise order is found in \cite{hofstad2005distancescmfinitevariance}.
 FPP and the weighted distances on CM with finite variance degrees are studied in \cite{bhamidi2010first, Bhamidi2017, bhamidi2014front}.

Also, it would be of interest to study the hopcount in more detail. For Erd\H{os}-R\'enyi random graphs (ERRG) with i.i.d.\ exponential weights on the edges, it is known that the hopcount is much larger order than the graph distance \cite{fppErdos2011}.
We conjecture that a similar result should hold for preferential attachment models with explosive edge weights, as the models contain a subgraph on at least $\sqrt{t}$ vertices (called the inner core in Section \ref{sec:upperbound}) that dominates a dense ERRG.
Moreover, we expect that our results on the tightness of the hopcount in Theorem \ref{thm:hopcount_nonexplosive} for weights of the form $L=1+X, I(X)<\infty$, should extend to the case where $I(X)=\infty$.
For this a better upper bound is necessary.
For the conservative class, extending Theorem \ref{thm:hopcount_nonexplosive} on the hopcount to weights that are not bounded away from zero is more difficult, and the hopcount within the dense inner core should be studied for this in more detail.
We also believe that \eqref{eq:tightness_cond} is not a necessary condition. However, we were not able to remove it in our proof of the upper bound.

Lastly, it would be interesting to study the geodesic, the least weighted path, in the neighbourhood of $u$ and $v$ in more detail.
Would it be possible to prove local weak limits of the \emph{geodesic} of the parts close to $u$ and $v$?
For CM, local weak limit theorems are established in \cite{dereichOrtgiese2018local_neighbourhoods_cm}.
We conjecture that using Theorem \ref{theorem:explosion_time_ppg} in the present paper and results from \cite{berger2014asymptotic,dereich2013random,garavaglia2019preferential},
similar results can be derived for PA. Another interesting question would be to analyse the age distribution of the vertices on the geodesic beyond the local neighbourhood of $u$ and $v$.

\section{Upper bound on the weighted distance}\label{sec:upperbound}
In this section we prove the upper bounds for Theorems \ref{thm:weighted_distance_nonexplosive} and \ref{thm:weighted_distance_explosive}, respectively. The upper bound of Theorem \ref{theorem:conservative_lwl} which follows from the same proof techniques, is postponed to Section \ref{sec:cons-lwl}. Recall $I(L)$ from \eqref{eq:il_sum_of_inverse_distribution}, and $Q_t$ from \eqref{eq:k_ast_qt}.
\begin{proposition}[Upper bound on the weighted distance, conservative case]\label{prop:upper_bound}
  Consider PA under the same conditions as Theorem \ref{thm:weighted_distance_nonexplosive}.
  Recall $I(L) = \infty$.
  Then for every $\delta, \varepsilon > 0$, when $t$ is sufficiently large
  \begin{equation}
    \prob\left(d_L^{(t)}(u,v) \leq (1+\varepsilon)2Q_t\right)
    \geq 1-\delta.\nonumber\label{eq:weighted_distance_nonexplosive_upperbound}
  \end{equation}
  Moreover, for the models FPA and VPA from Definition \ref{def:pa_m_delta} and \ref{def:pa_gamma},
  if $F_L$ satisfies \eqref{eq:tightness_cond}, there exists a constant $M_{\ref{prop:lowerbound_unexplosive}}=M_{\ref{prop:lowerbound_unexplosive}}(\delta)$ such that for $t$ sufficiently large
  \[
  \prob\left(d_L^{(t)}(u,v) \leq 2Q_t + 2M_{\ref{prop:lowerbound_unexplosive}}\right) \geq 1-\delta.
  \]
\end{proposition}
\begin{proposition}[Upper bound on the weighted distance, explosive case]\label{prop:upperbound_explosive}
  Consider PA under the same conditions as Theorem \ref{thm:weighted_distance_explosive}.
    Recall $I(L) < \infty$.
Then, there is a coupled probability space, such that for
  every $\delta,\varepsilon > 0$ there exists a constant $N\in\mathbb{N}$ such that for $t$ sufficiently large
  \begin{equation}
    \prob\left(d_L^{(t)}(u,v)\leq \beta^{\mathrm{LWL}^{(u)}}_{N} + \beta^{\mathrm{LWL}^{(v)}}_{N} + \varepsilon\right)
    \geq 1-\delta\label{eq:weighted_distance_explosive_upperbound},
  \end{equation}
  where $\mathrm{LWL}^{(u)}, \mathrm{LWL}^{(v)}$ are the LWL trees coupled to the neighbourhood of $u,v$, respectively.
\end{proposition}
\noindent
\begin{figure}
\includegraphics[width=0.9\textwidth]{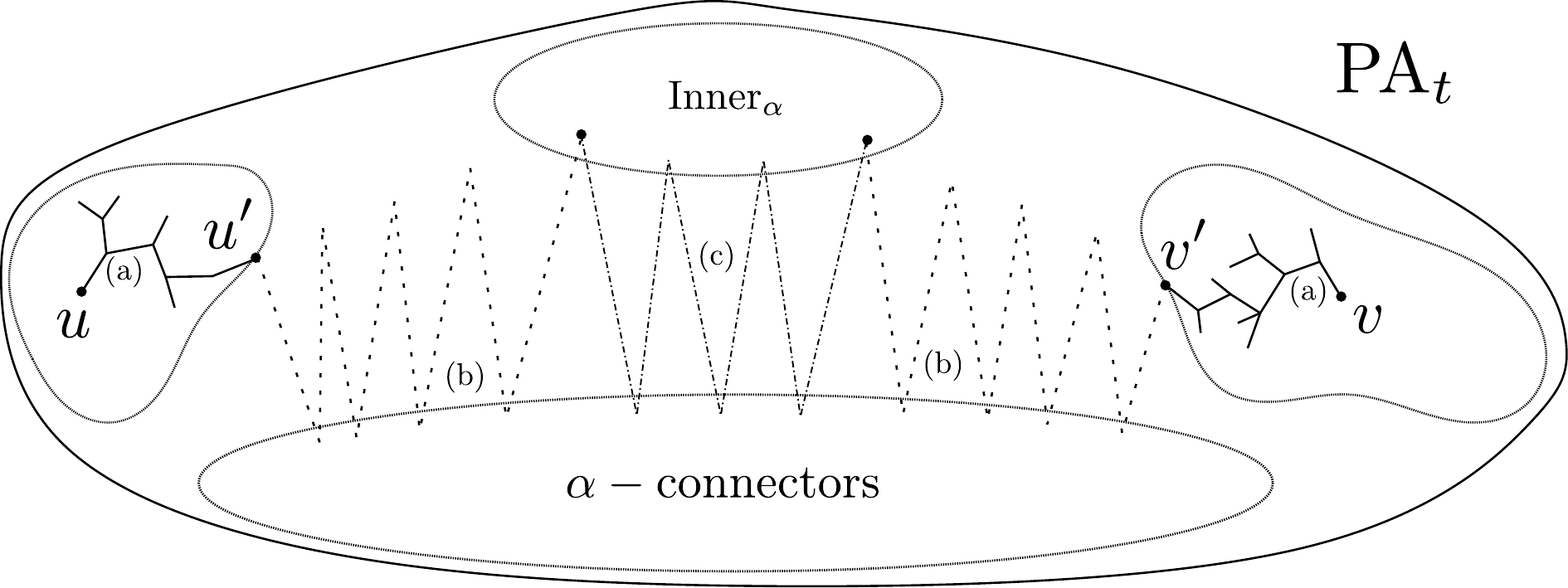}
\caption{The constructed five-segment path from $u$ to $v$, via vertices with sufficiently high degree $u'$ and $v'$ and $\mathrm{Inner}_\alpha$.}
\label{fig:constructed-path}
\end{figure}
Throughout this section, we look at the graph at times $t$ and $t':=\alpha t$,
for some $\alpha\in[\frac{1}{2}, 1)$. For the upper bound on $d_L^{(t)}(u,v)$ it is enough to construct a path between $u$ and $v$,
and study its weight.
The path that we construct, consists of five segments of three different types.
First, for $q\in\{u,v\}$, we construct  a path, consisting of one segment of type (a) and one of type (b), to
\begin{equation}
\mathrm{Inner}_{\alpha}:=\{w\in[\alpha t]:D_{\alpha t}(w)\geq {(\alpha t)}^{\frac{1}{2(\tau-1)}}\log(\alpha t)^{-\frac{1}{2}}\},\label{eq:innercore}
\end{equation}
i.e., vertices with a very large degree, also called inner core, see Figure \ref{fig:constructed-path}. The total weight of one path contributes almost
half of the total weight of the entire path from $u$ to $v$.
The segment of type (a) connects $q\in\{u,v\}$ to a vertex $q'$ that has degree at least $s_0\in\mathbb{N}$. It only passes  through vertices that arrived before time $\alpha t$, \emph{old} vertices.
To do so, we shall condition on $q<\alpha t$, which happens w/p close to one if $\alpha$ is close to one.
The segment of type (b) connects $q'$ to the inner core and alternatingly passes through old vertices and \emph{$\alpha$-connectors}, vertices that arrived after time $\alpha t$. Similarly for segment type (c), see Figure \ref{fig:constructed-path}, we construct a path with negligible total weight
that connects two vertices in the inner core, by alternatingly using $\alpha$-connectors and other vertices in the inner core.
By construction, all edges on segments of type (a) arrived \emph{before} time $\alpha t$, while on types (b) and (c) all edges arrived \emph{after} $\alpha t$.

We present now three segment-specific propositions, used in both the explosive and conservative case.
Afterwards we introduce some necessary notation to construct the path.
Lastly, we show how these propositions together prove Propositions \ref{prop:upper_bound} and \ref{prop:upperbound_explosive}.
Starting with segment (a), we show that the number of edges on the path between $q$ and $q'$ is bounded, for $q\in\{u,v\}$. In the conservative case, its total weight is negligible compared to $Q_t$.
    In the explosive case, this part is the main contributor and later we show that its total weight tends in distribution to the (finite) explosion time of the LWL.
\begin{proposition}[Bounded graph distance to a vertex with degree at least $s$]\label{prop:graph_distance_to_high_degree}
  Consider PA with power-law exponent $\tau\in(2, 3)$, i.i.d.\ weights on the edges with distribution $F_L$
  and fix $\delta_{\ref{prop:graph_distance_to_high_degree}}>0$.
  Let $q$ be chosen uniformly at random from $[t]$.
  For any $s_{\ref{prop:graph_distance_to_high_degree}}\in\mathbb{N}$, there is a constant
  $C_{\ref{prop:graph_distance_to_high_degree}}=C_{\ref{prop:graph_distance_to_high_degree}}(s_{\ref{prop:graph_distance_to_high_degree}}, \delta_{\ref{prop:graph_distance_to_high_degree}})$
  such that for $t$ sufficiently large
  \begin{equation}
    \prob\bigg(
      \bigcap_{q'\in[t]: D_{t}(q')\geq s_{\ref{prop:graph_distance_to_high_degree}}}\left\{d_G^{(t)}(q, q')
      \geq
      C_{\ref{prop:graph_distance_to_high_degree}}\right\}
    \bigg)
    \leq
    \delta_{\ref{prop:graph_distance_to_high_degree}}.\nonumber%\label{eq:bounded_distance_to_t}
  \end{equation}
  We denote the complement of the above event between brackets by
  $\mathcal{E}_{\ref{prop:graph_distance_to_high_degree}}^{(t)}(q, s_{\ref{prop:graph_distance_to_high_degree}})$.
\end{proposition}
\noindent
The proof for FPA follows from a minor adaptation of the proof of \cite[Theorem 3.6]{dommers2010diameters}.
For GVPA it follows from an adaptation of \cite[Proposition 5.10]{monch2013distances}.
We refer the reader to the cited paper and thesis to fill in the details.
We emphasize that we apply Proposition \ref{prop:graph_distance_to_high_degree} at time $\alpha t$, rather than $t$.

From the vertex $q'$ with degree at least $s_0$ at time $\alpha t$, we construct a path to the inner core, corresponding to segment  (b). We show that there are many such paths, allowing to bound the weight.
The next proposition is the main (technical) contribution of the paper.
Due to this statement, we obtain tight bounds for the various distances in FPA, improving upon existing results \cite{caravenna2016diameter,dereich2012typical}.
Its proof can easily be adapted to obtain tight fluctuations on the typical weighted distance in the configuration model if condition \eqref{eq:tightness_cond} is satisfied, improving results in \cite{adriaans2017weighted}.
\begin{proposition}[{Upper bound on the weighted distance to $\mathrm{Inner}_{\alpha}$}]\label{prop:weighted_distance_in_t}
  Consider PA with power-law exponent $\tau\in(2, 3)$, i.i.d.\ weights on the edges with distribution $F_L$.
  Fix $\delta_{\ref{prop:weighted_distance_in_t}},\varepsilon_{\ref{prop:weighted_distance_in_t}}>0$, $\alpha\in[1/2,1)$.
  There exists $s_0=s_0(\delta_{\ref{prop:weighted_distance_in_t}})\in\mathbb{N}$ and a constant $M_{\ref{prop:weighted_distance_in_t}}$, such that for $s>s_0$, and any $q'\in[\alpha t]$ with $D_{\alpha t}(q')= s$,
    if $t$ is sufficiently large, for FPA or VPA, if $F_L$ satisfies \eqref{eq:tightness_cond},
  \begin{equation}
    \prob\bigg(d_L^{(t)}(q', \mathrm{Inner}_{\alpha}) \geq
    M_{\ref{prop:weighted_distance_in_t}} +
    \sum_{k=\lfloor h_\tau(s)\rfloor}^{K^\ast_t + \lfloor h_\tau(s) \rfloor + 4}F_{L}^{(-1)}
        \left(\exp\left(-(\tau-2)^{-k/2}\right)
                  \right)
                  \bigg)
    \leq \delta_{\ref{prop:weighted_distance_in_t}},\label{eq:weighted_distance_in_t}
  \end{equation}
  where $h_\tau(s)=2\log\log(s)/|\log(\tau-2)|+c_{\tau}$ for some constant $c_{\tau}$.
  Without \eqref{eq:tightness_cond}, it holds that
  \begin{equation}
    \prob\bigg(d_L^{(t)}(q', \mathrm{Inner}_{\alpha}) \geq
    (1+\varepsilon_{\ref{prop:weighted_distance_in_t}})
    \sum_{k=\lfloor h_\tau(s)\rfloor}^{K^\ast_t + \lfloor h_\tau(s) \rfloor + 4}F_{L}^{(-1)}
        \left(\exp\left(-(\tau-2)^{-k/2}\right)
                  \right)
                  \bigg)
    \leq \delta_{\ref{prop:weighted_distance_in_t}},\label{eq:weighted_distance_in_t_general},
  \end{equation}
where the constant $c_\tau$ in the function $h_\tau$ might be different and can depend on $\varepsilon_{\ref{prop:weighted_distance_in_t}}$.
\end{proposition}
\noindent This part of the path is the main contributor to the upper bound in the conservative case. If $I(L)<\infty$, it follows that the value of the sum in \eqref{eq:weighted_distance_in_t} can be made arbitrarily small by increasing $s$, because $h_\tau(s)$ tends to infinity.
For the conservative case, comparing the sum in $Q_t$ in \eqref{eq:k_ast_qt} to the sum in \eqref{eq:weighted_distance_in_t}, one sees that they are identical up
to a shift of the summation boundaries.

In the next proposition we bound the graph \emph{and} weighted distance within the inner core,  segment  (c) in Figure \ref{fig:constructed-path}.
\begin{proposition}[Inner core has negligible weighted distance]\label{prop:innercore_bounded}
  Consider PA with power-law exponent $\tau\in(2, 3)$, i.i.d.\ weights on the edges with distribution $F_L$.
  Recall  $\mathrm{Inner}_{\alpha}$ from \eqref{eq:innercore} and fix $\delta_{\ref{prop:innercore_bounded}}>0, \alpha\in[1/2, 1)$.
  Then there exists $C_{\ref{prop:innercore_bounded}}>0$, such that for all  sufficiently large $t$, and for any two fixed vertices $w_1, w_2$  in $\mathrm{Inner}_{\alpha}$,
  \begin{equation}
    \prob\left(d_G^{(t)}(w_1, w_2)\geq C_{\ref{prop:innercore_bounded}} \right)
    \leq
    \delta_{\ref{prop:innercore_bounded}}.\label{eq:bounded_distance_inner_graph}
  \end{equation}
   Moreover, if $F_L$ satisfies $F_L(x)>0$ for all $x>0$\footnote{This constraint holds whenever $I(L)<\infty$. However, it might not hold when $I(L)=\infty$.},
  then for any $\varepsilon_{\ref{prop:innercore_bounded}}>0$, if $t$ is sufficiently large,
  \begin{equation}
    \prob\left(d_L^{(t)}(w_1, w_2)\geq \varepsilon_{\ref{prop:innercore_bounded}} \right)
    \leq
    \delta_{\ref{prop:innercore_bounded}}.\label{eq:bounded_distance_inner}
  \end{equation}

\end{proposition}
\noindent The proof of this proposition is deferred to the appendix on page \pageref{proof:innercore_bounded}.
It makes partly use of the same concepts as the proof of Proposition \ref{prop:weighted_distance_in_t},
combined with a coupling argument to a dense Erd\H{o}s-R\'enyi (ER) graph. In this ER-graph we show that there are many disjoint paths connecting
$w_1$ and $w_2$, allowing to bound $d_L^{(t)}(w_1, w_2)$.

Before outlining the proof of Proposition \ref{prop:weighted_distance_in_t},
we state the formal definitions of notions that are of particular importance throughout the proof.
\begin{definition}[Layers, $\alpha$-connectors, and the greedy path]\label{def:layers_connectors_greedy}
  Fix $\alpha\in[1/2, 1)$. Let
  \begin{align}
    s_{k} &:= \min\left\{s_{k-1}^{(1-\varepsilon_{k-1})\left(\tau-2\right)^{-1}}, \,(\alpha t)^{\frac{1}{2(\tau-1)}}\log(\alpha t)^{-\frac{1}{2}}\right\},\qquad k\in \mathbb{N},\label{eq:sk}\\
    K_t &:= \min  \left\{k: s_k\geq (\alpha t)^{\frac{1}{2(\tau-1)}}\log(\alpha t)^{-\frac{1}{2}}\right\}, \label{eq:kt}
  \end{align}
  where $s_0>1$ and $(\varepsilon_k)_{k\geq 0}$ is a sequence that tends to 0 for FPA and VPA under condition \eqref{eq:tightness_cond}, and $\varepsilon_k\equiv \varepsilon_G$
  for some small constant $\varepsilon_G>0$ otherwise.
  For  $(s_k)_{k\geq 0}$, we define the \emph{$k$-th layer} as
  \begin{equation}
    \mathcal{L}_k:=\{x\in[\alpha t]: D_{\alpha t}(x) \geq s_k\}.\nonumber%\label{eq:layers}
  \end{equation}
  A vertex $y$ in $[t]\backslash[\alpha t]$ is called an \emph{$\alpha$-connector} of $(x,z)$ if it
  is connected both to $x$ and $z$.
  Let
  $
  \mathcal{A}_k(x):=\{(y,z)\in [t]\backslash[\alpha t]\times\mathcal{L}_k:
  x\leftrightarrow y\leftrightarrow z \}.
  $
  For a vertex $\pi_0:=q'\in\mathcal{L}_0$, we construct a greedy path
  $\pi^{\mathrm{gr}}=(\pi_0, y_{1}, \pi_{1}, y_{2}, \pi_{2},..., y_{K_t-1}, \pi_{K_t})$
  of length $2K_t$ by sequentially choosing
  \begin{equation}
  (y_k, \pi_k)=\argmin_{(y, z)\in\mathcal{A}_k(\pi_{k-1})}\left\{L_{(\pi_{k-1}, y)}+L_{(y, z)}\right\}\nonumber%\label{eq:greedy_path}
  \end{equation}
  if it exists. If $\mathcal{A}_k(\pi_{k-1})=\varnothing$, we say that the construction of the greedy path fails at step $k$.
\end{definition}
\noindent Note that on the greedy path $\pi^{\mathrm{gr}}$, $\pi_k\in\mathcal{L}_k$ and $y_k$
is an $\alpha$-connector of $(\pi_k,\pi_{k+1})$.
\subsubsection*{Outline of the proof of Proposition \ref{prop:weighted_distance_in_t}}The idea of Proposition \ref{prop:weighted_distance_in_t} is to construct a greedy
path to the inner core and use its total weight as an upper bound for the actual shortest path.
We outline the proof for FPA and VPA under condition \eqref{eq:tightness_cond}. The other cases follow by similar steps.
\begin{enumerate}[label=(\Roman*),leftmargin=2\parindent]
  \item The greedy path consists of $K_t$ cherries, i.e., of $2K_t$ edges. Recall $K^\ast_t$ from \eqref{eq:k_ast_qt}. We show for a specific choice of $\varepsilon_k$
  that
  \begin{equation}
  2K_t \,\leq\, K^\ast_{t} + 4.
  \nonumber%\label{eq:upperbound_outline_kt_bound}
  \end{equation}
  \item We show that the sizes of the sets $\mathcal{A}_k(\pi_{k-1})$ are bounded from below by a doubly exponentially growing sequence
$(n_k)_{k\geq 0}$, i.e., for $\delta>0$, there is an $s_0$ such that for large $t$
\begin{equation}
\prob\bigg(\bigcup_{k\in[K_t]}\left\{|\mathcal{A}_{k+1}(\pi_k)|\leq n_k\right\}\bigg)
\leq\delta. \label{eq:small_neighbourhoods_error_outline}
\end{equation}
As a result, the greedy path actually exists, w/p at least $1-\delta$.
To prove this, the choice of the exponent of $(s_k)_{k\geq0}$ in \eqref{eq:sk} is crucial,
and in particular the choice of $(\varepsilon_k)_{k\geq 0}$.
In \cite[Theorem 3.1]{dommers2010diameters} and \cite[Proposition 3.1]{dereich2012typical} similar constructions
of greedy paths are used. In those proofs, the exponent of $(s_k)_{k\geq0}$ is equal to $1/(\tau-2)$
and every term is \emph{corrected with a $\log(t)$-term} to ensure that every vertex in layer $\mathcal{L}_k$ has
at least \emph{one} $t$-connector.
In contrary, we \emph{correct the exponent} by $-{\varepsilon_k}/(\tau-2)$,
implying that $(s_k)_{k\geq 0}$ grows slower.
In return, every vertex in $\mathcal{L}_k$ has \emph{many} $t$-connectors whp, that allows for small weighted distances.
\item By the construction of $\pi^{\mathrm{gr}}$, the weighted distance between $\pi_0$ and $\pi_{K_t}$
can be bounded by
\begin{equation}
d_L(q', \pi_{K_t})\leq
\sum_{k\in[K_t]}\min_{(y,z)\in[\mathcal{A}_k(\pi_{k-1})]}\left\{L_{(\pi_{k-1}, y)} + L_{(y,z)}\right\}.
\label{eq:weighted_distance_in_t_sum_tk}
\end{equation}
As the minimum is non-decreasing if we consider less elements,
conditionally on the complement of the event in \eqref{eq:small_neighbourhoods_error_outline}, we weaken
the bound to
\begin{equation}
d_L(q', \pi_{K_t})\leq \sum_{k\in[K_t]}\min_{j\in[n_k]}\left\{L^{(k)}_{j1} + L^{(k)}_{j2}\right\}.
\label{eq:weighted_distance_in_t_sum_nk}
\end{equation}
We show that the generalized inverse $F_{L_1 + L_2}^{(-1)}$ can be related to the generalized inverse $F_L^{(-1)}$. This allows us
to bound the rhs of \eqref{eq:weighted_distance_in_t_sum_nk} and obtain the asserted bound
\eqref{eq:weighted_distance_in_t} from Proposition \ref{prop:weighted_distance_in_t}
as we make the error probabilities arbitrarily small by choosing $s_0$ sufficiently large.
\end{enumerate}
We start with Step (I) by proving an upper bound for $K_t$.
\begin{lemma}\label{lemma:kt_kast}
  Let $K_t$ be as in \eqref{eq:kt}, $K^\ast_{t}$ as in \eqref{eq:k_ast_qt}, $\tau\in(2,3)$, $\alpha\in[1/2,1)$. For FPA and VPA, if \eqref{eq:tightness_cond} holds, set
  \begin{equation}
    \varepsilon_k :=(k+2)^{-2}, \qquad k\in\mathbb{N}. \nonumber%\label{eq:epsilon_k}
  \end{equation}
  There exists a constant $c_{\ref{lemma:kt_kast}}>0$ such that for $s_0$ sufficiently large
  \begin{align}
    s_k \geq s_0^{c_{\ref{lemma:kt_kast}}(\tau-2)^{-k}}, \qquad
    2K_t \leq K^\ast_t+4.\label{eq:claim_sk_kt_kast}
  \end{align}
  For FPA and VPA if \eqref{eq:tightness_cond} does not hold and GVPA, set $\varepsilon_G>0$ such that
  \begin{equation}
     \frac{\log\big(1/(\tau-2)\big)}{\log\big((1-\varepsilon_G)/(\tau-2)\big)} = 1 + \varepsilon_{\ref{prop:weighted_distance_in_t}}. \label{eq:epsilon_G}
  \end{equation}
  Then for $s_0$ sufficiently large
  \begin{equation}
    2K_t \leq (1+\varepsilon_{\ref{prop:weighted_distance_in_t}})K^\ast_t + 4.
    \label{eq:claim_sk_kt_kast_general}
  \end{equation}
  \begin{proof}
    First we consider FPA and VPA under condition \eqref{eq:tightness_cond}.
    Recall the definition of $(s_k)_{k\geq 0}$ from \eqref{eq:sk}. By iterating the recursion, we obtain
    \[
    s_k = s_0^{\prod_{j=0}^{k-1}\left((1-\varepsilon_j)/(\tau-2)\right)}.
    \]
    By our choice of $\varepsilon_k=1/(k+2)^2$, the product $\prod_{j=1}^\infty (1-\varepsilon_j)>0$, which yields the first bound in \eqref{eq:claim_sk_kt_kast} for some constant $c_{\ref{lemma:kt_kast}}>0$.
    Hence, by the definition of $K_t$ in \eqref{eq:kt},
    \[
    K_t \leq \min\left\{k:s_0^{c_{\ref{lemma:kt_kast}}(\tau-2)^{-k}}\geq (\alpha t)^{\frac{1}{2(\tau-1)}}\log(\alpha t)^{-\frac{1}{2}} \right\}.
    \]
    By taking logarithms twice in the inequality between brackets above, we obtain
    \begin{equation}\label{eq:kt111}
    K_t\leq\left\lceil \frac{\log\left(\frac{1}{2(\tau-1)}\log(\alpha t) - \frac{1}{2}\log\log(\alpha t)\right) - \log\log(s_0)-\log(c)}{|\log(\tau-2)|}\right\rceil.
    \end{equation}
    If $s_0$ is sufficiently large, the numerator of $K_t$ above is smaller than $\log\log(t)$ for $t$ sufficiently large.
    Bounding the rounding operations yields by the definition of $K^\ast_t$ in \eqref{eq:k_ast_qt}
    \[
    2K_t \leq 2\frac{\log\log(t)}{|\log(\tau-2)|} + 2 \leq K^\ast_t+3.
    \]
    For proving \eqref{eq:claim_sk_kt_kast_general} for GVPA, and FPA and VPA without condition \eqref{eq:tightness_cond}, we use a similar reasoning. By taking logarithms twice in the definition of $K_t$ in \eqref{eq:kt}, we obtain a similar formula to that in \eqref{eq:kt111}, and after bounding the numerator as before, as well as using the implicit definition of $\varepsilon_G$ in \eqref{eq:epsilon_G} and that of $K^\ast_t$ in \eqref{eq:k_ast_qt}, we arrive to
    \[
    2K_t
    \,\leq\,
    (1+\varepsilon_{\ref{prop:weighted_distance_in_t}})K^\ast_t + 3 + \varepsilon_{\ref{prop:weighted_distance_in_t}}
    \,\leq\, (1+\varepsilon_{\ref{prop:weighted_distance_in_t}})K^\ast_t + 4,
    \]
    finishing the proof. \end{proof}
\end{lemma}

We recall two preliminary lemmas from \cite{dommers2010diameters}
that help us control the error probability in \eqref{eq:small_neighbourhoods_error_outline}.
\begin{lemma}[Probability on being a $\alpha$-connector for an arbitrary set \cite{dommers2010diameters}]\label{lemma:t_connector_probability}
    Consider PA under the same conditions as Proposition \ref{prop:weighted_distance_in_t}.
Let $\alpha\in[1/2,1)$. For $x\in[\alpha t]$, a set $\mathcal{V}\subset[\alpha t]$, conditionally
on $\mathrm{PA}_{\alpha t}$, the probability that $y\in[t]\backslash[\alpha t]$ is an $\alpha$-connector
of $(x,\mathcal{V})$ is at least
\begin{equation}
   \frac{\eta_{\ref{lemma:t_connector_probability}} D_{\alpha t}(x)D_{\alpha t}(\mathcal{V})}{(\alpha t)^2}
=:p_{\alpha t}(x,\mathcal{V}),\label{eq:t_connector_probability}
\end{equation}
where $\eta_{\ref{lemma:t_connector_probability}}>0$ is a constant, and $D_{\alpha t}(\mathcal{V}):=\sum_{z\in \mathcal{V}}D_{\alpha t}(z)$.
Moreover, w/p at least $p_{\alpha t}(x,\mathcal{V})$, the event $\{y \text{ is an $\alpha$-connector of } (x,\mathcal{V})\}$ happens independently of other vertices in $[t]\backslash[\alpha t]$.
\end{lemma}
\noindent We use the above lemma for layers $\mathcal{V}=\mathcal{L}_k$ and a vertex $x\in\mathcal{L}_{k-1}$.
The following lemma allows us to lower bound the total degree of vertices in $\mathcal{L}_k$,
$D_{\alpha t}\left({\mathcal{L}_k}\right)$.
Although it assumes for the model $\mathrm{GVPA}(f)$ that the function $f(x)$ needs to be affine for $x$ sufficiently large,
we show below that this restriction does not propagate to the requirements of Proposition \ref{prop:weighted_distance_in_t}.
\begin{lemma}[Impact of high degree vertices {\cite[Lemma A.1]{dommers2010diameters}}, {\cite[Theorem 1.1(a)]{dereich2009random}}]\label{lemma:total_degree_high_degree_vertices} Let $PA_t$
satisfy the same conditions as Proposition \ref{prop:weighted_distance_in_t}, and additionally for
  $\mathrm{GVPA}(f)$, assume that the function $f(x)$ is affine for all $x$ larger than some $x_0\in\mathbb{R}$.
  Let $\alpha\in[1/2,1)$, and
  $\phi\in\mathbb{R}$ satisfy $x_0\leq\phi\leq (\alpha t)^{\frac{1}{2(\tau-1)}}(\log (\alpha t))^{-\frac{1}{2}}$. There exists a
  constant $c_{\ref{lemma:total_degree_high_degree_vertices}}>0$ such that
  \begin{equation}
    \prob\bigg(\sum_{z:D_{\alpha t}(z)\geq \phi}D_{\alpha t}(z) \geq c_{\ref{lemma:total_degree_high_degree_vertices}}\alpha t\phi^{2-\tau}\bigg) = 1-o(t^{-1}).
    \nonumber%\label{eq:total_degree_high_degree_vertices}
  \end{equation}
  Moreover, whp the number of vertices with degree at least $\phi$ is at least $\sqrt{\alpha t}$.
\end{lemma}
\noindent Our version of the above lemma is slightly different from \cite[Lemma A.1]{dommers2010diameters},
as we do not assume that $\phi=\phi(t)$ tends to infinity with $t$,
while \cite{dommers2010diameters} does.
We refer the reader to the proof of \cite[Lemma A.1]{dommers2010diameters}
to see that the proof is also valid for constant $\phi$. Also \cite[Theorem 1.1(a)]{dereich2009random}
is slightly different from the statement here, as it states convergence of the degree distribution in total variation norm.
However, especially in combination with \cite[Example 3.1]{dereich2009random}, it is easy to check Lemma \ref{lemma:total_degree_high_degree_vertices} is an immediate corollary.
We continue with the main lemma of Step (II).
 Recall $K_t$, $(s_k)_{k\geq 0}$, and $\{\mathcal{L}_k\}_{k\geq 0}$
  from Definition \ref{def:layers_connectors_greedy} and $\varepsilon_k$ from Lemma \ref{lemma:kt_kast}.
\begin{lemma}[Lower bound on the number of {$\alpha$-connectors}]\label{lemma:error_prob_tk}
  Consider FPA or VPA under the same conditions as Proposition \ref{prop:weighted_distance_in_t}.
  Let $\alpha\in[1/2,1)$.
There exists constant $c_{\ref{lemma:error_prob_tk}}>0$, $c'>0$,
  such that for an arbitrary set $\{\pi_0, ..., \pi_{K_t-1}\}$, where $\pi_k\in\mathcal{L}_k$,
  and $s_0, t$ sufficiently large
  \begin{equation}
  \prob\bigg(\bigcup_{k\in[K_t]}\left\{|\mathcal{A}_k(\pi_{k-1})|\leq c_{\ref{lemma:error_prob_tk}} s_{k-1}^{{\varepsilon_{k-1}/2}}\right\}\bigg)
  \leq 2\exp\left(-\frac{c_{\ref{lemma:error_prob_tk}}}{4}s_0^{c'}\right)=:\delta_{\ref{lemma:error_prob_tk}}^{(s_0)}.\label{eq:error_prob_tk}
  \end{equation}
  \begin{proof}
    First, we show a stochastic domination argument of GVPA to VPA. Afterwards, we prove the existence of a binomial random variable $A$ that
    is dominated by $|\mathcal{A}_k(\pi_{k-1})|$. Lastly, we apply Chernoff's bound to $A$ and
    show that the result follows.

    In order to apply Lemma \ref{lemma:total_degree_high_degree_vertices}, $f(x)$ must be an affine function for large $x$. Recall $\gamma_f$ from \eqref{eq:gamma_f} and assume $f$ is non-affine.
    For any $\gamma_{f_1}\in(1/2,\gamma_f)$, there are $x_0\in\mathbb{N}$, $\eta\in\mathbb{R}$, such that
    \[
    f(x)\geq f_1(x):=
    \begin{dcases}
      f(x)&x\leq x_0\\
      \gamma_{f_1}x + \eta&x>x_0,
    \end{dcases}
    \]
    which is affine for $x>x_0$, and still concave.
    Hence, the model $\mathrm{GVPA}(f_1)$ is well-defined and is stochastically dominated by $\mathrm{GVPA}(f)$.
    Assume that $s_0>x_0$ so that we can apply Lemma \ref{lemma:total_degree_high_degree_vertices} on the model GVPA($f_1$).
    As $\gamma_{f_1}$ can be chosen arbitrarily close to $\gamma_{f}$, we can choose it such that the power-law exponent $\tau_{f_1}$
    is close to $\tau_f$, in particular so that the inequality
    \begin{equation} \label{eq:stochastic_domination_epsilon_bound}
    \frac{1-\varepsilon_G/2}{1-\varepsilon_G} \geq \frac{\tau_{f_1}-2}{\tau-2}
  \end{equation}
  holds with $\varepsilon_G$ from \eqref{eq:epsilon_G}. Assume that $s_0>x_0$, and write $\varepsilon_k\equiv \varepsilon_G$ for the model GVPA. Define
    \[
    \tau' :=
    \begin{dcases}
      \tau & \text{for FPA, VPA,}\\
      \tau_{f_1} &\text{for GVPA}.
    \end{dcases}
    \]
    Let $\pi_{k-1}\in\mathcal{L}_{k-1}$ for some $k\in[K_t]$.
    By Lemma \ref{lemma:t_connector_probability}, the probability that
    $y\in[t]\backslash[\alpha t]$ is an $\alpha$-connector of $(\pi_{k-1}, \mathcal{L}_k)$
    is at least $p_{\alpha t}(\pi_{k-1}, \mathcal{L}_k)$, independently of other vertices
    in $[t]\backslash[\alpha t]$, see \eqref{eq:t_connector_probability}.
    Since there are in total $(1-\alpha)t$ possible $\alpha$-connectors,
    the random variable $|\mathcal{A}_k(\pi_{k-1})|$ stochastically dominates
    a binomial random variable, i.e.,
    \begin{equation}
      |\mathcal{A}_k(\pi_{k-1})|\overset{d}{\geq}\text{Bin}\left((1-\alpha) t, p_{\alpha t}(\pi_{k-1}, \mathcal{L}_k)\right)=:A_k.
      \label{eq:tk_dominated}
    \end{equation}
    Here, the notation $\overset{d}\geq$ is used for stochastic domination.
    Conditioning on $\mathcal D_{\mathcal{L}_k}(\alpha t)$
    yields by Lemma \ref{lemma:total_degree_high_degree_vertices} for $\tau'$,
    \begin{equation}
    \E[A_k]
    \geq \E\big[A_k \mid D_{\mathcal{L}_k}(\alpha t)\geq c_{\ref{lemma:total_degree_high_degree_vertices}}\alpha ts_k^{2-\tau'}\big]\prob\left(D_{\mathcal{L}_k}(\alpha t)\geq c_{\ref{lemma:total_degree_high_degree_vertices}}\alpha ts_k^{2-\tau'}\right),
    \nonumber    \end{equation}
    where the latter factor equals $1-o(t^{-1})$.
    Since $\pi_{k-1}\in\mathcal{L}_{k-1}$ and thus $D_{\alpha t}(\pi_{k-1})\geq s_{k-1}$ by the construction of $\mathcal{L}_{k-1}$
    in Definition \ref{def:layers_connectors_greedy}, we substitute the value
    of $p_{\alpha t}(\pi_{k-1},\mathcal{L}_k)$ in \eqref{eq:t_connector_probability} to
    bound the expectation of the binomial random variable $A_k$ further to obtain
    \begin{align}
    \E[A_k]
    &\geq (1-\alpha)t\frac{\eta_{\ref{lemma:t_connector_probability}} c_{\ref{lemma:total_degree_high_degree_vertices}}\alpha ts_k^{2-\tau'}s_{k-1}}{(\alpha t)^2}\left(1-o(t^{-1})\right)
    \geq 2c_{\ref{lemma:error_prob_tk}}s_k^{2-\tau'}s_{k-1} \nonumber\\
    &\geq 2c_{\ref{lemma:error_prob_tk}}s_{k-1}^{1 - (2-\tau')\left(\frac{1-\varepsilon_{k-1}}{\tau-2}\right)}
    \geq 2c_{\ref{lemma:error_prob_tk}}s_{k-1}^{\varepsilon_{k-1}/2}\label{eq:lowerbound_etk}
    \end{align}
    for some constant $c_{\ref{lemma:error_prob_tk}}\in(0,(1-\alpha)c_{\ref{lemma:total_degree_high_degree_vertices}}\eta_{\ref{lemma:t_connector_probability}}/(2\alpha))$ if $t$ is sufficiently large,
    by the recursive definition of $s_k$ in \eqref{eq:sk}. The last inequality is a consequence of \eqref{eq:stochastic_domination_epsilon_bound}.
Next we apply Chernoff's bound, see e.g.\ \cite{mitzenmacher2005probability}, in the following form: for $\psi_k>0$,
    \begin{equation}
    \prob(A_k\leq(1-\psi_k)\E[A_k])\leq \exp\big(-\psi_k^2\E[A_k]/2\big).
    \nonumber%\label{eq:chernoff}
    \end{equation}
    Choosing $\psi_k=(1-c_{\ref{lemma:error_prob_tk}}s_{k-1}^{{\varepsilon_{k-1}/2}}/\E[A_k])$,
    yields by \eqref{eq:lowerbound_etk} that $\psi_k\geq 1/2$.
    Hence, we can bound $\psi_k^2\E[A_k]\geq c_{\ref{lemma:error_prob_tk}}s_{k-1}^{{\varepsilon_{k-1}/2}}/2$,
    so that
    \[
    \prob\left(A_k\leq c_{\ref{lemma:error_prob_tk}} s_{k-1}^{\varepsilon_{k-1}/2}\right)\leq \exp\left(-c_{\ref{lemma:error_prob_tk}}s_{k-1}^{\varepsilon_{k-1}/2}/4\right).
    \]
    Applying a union bound over $k\in[K_t]$ and switching back to the
    dominating random variable $|\mathcal{A}_k(\pi_{k-1})|$ as in \eqref{eq:tk_dominated} results in
    \begin{equation}
    \prob\Big(\bigcup_{k\in [K_t]}\left\{|\mathcal{A}_k(\pi_{k-1})|\leq c_{\ref{lemma:error_prob_tk}} s_{k-1}^{\varepsilon_{k-1}/2}\right\}\Big)
    \leq \sum_{k\in [K_t]} \exp\left(-c_{\ref{lemma:error_prob_tk}}s_{k-1}^{\varepsilon_{k-1}/2}/4\right).
    \label{eq:error_prob_tk_error_prob_sum}
    \end{equation}
    Because $\varepsilon_k$ is a constant for GVPA, the asserted bound in \eqref{eq:error_prob_tk} follows immediately for $s_0$ sufficiently large.
    For FPA and VPA, it remains to bound the sum on the rhs. We apply the lower bound on $s_k$ from \eqref{eq:claim_sk_kt_kast} and observe that
   $(\tau-2)^{-k}$ grows much faster than $(k+2)^2$, so
    \[
    s_{k-1}^{{\varepsilon_{k-1}/2}}\ge s_0^{c_{\ref{lemma:kt_kast}}(\tau-2)^{-k}(k+2)^{-2}} \geq s_0^{c'c^{k}},
    \]
    for some $c,c'>0$, whence the asserted bound \eqref{eq:error_prob_tk} follows for all sufficiently large $s_0$.
  \end{proof}
\end{lemma}

\noindent The above lemma ensures that there are \emph{many} $\alpha$-connectors.
However, we still need to bound the probability that the weighted
distance between $\pi_{k-1}$ and $\mathcal{L}_{k}$ is sufficiently small, given
that there are \emph{enough} $\alpha$-connectors, as described in Step (III) of the outline.
\begin{lemma}[Minimum of i.i.d.\ random variables]\label{lemma:minimum_iid_rvs}
  Let $L_1, ..., L_n$ be i.i.d.\ random variables having distribution $F_L$.
  Then for all $\xi>0$
  \begin{align}
  \prob\Big(\min_{j\in[n]} L_j \geq F_L^{(-1)}\left(n^{-1+\xi}\right)\Big)
  \overset{(\star)}\leq \mathrm{e}^{-n^{\xi}}, \qquad%\label{eq:minimum-iid-upper}\\
  \prob\Big(\min_{j\in[n]} L_j \leq F_L^{(-1)}\left(n^{-1-\xi}\right)\Big)
  \overset{(\ast)}\leq
  n^{-\xi}.\label{eq:minimum-iid-lower}
  \end{align}
  \begin{proof}
    Since the random variables are i.i.d.,
    \begin{align*}
      \prob\Big(\min_{j\in[n]} L_j \geq z(n)\Big) &= (1-F_L(z(n)))^n,
    \end{align*}
    We substitute $z(n)=F_L^{(-1)}\left(n^{-1\pm\xi}\right)$, so that
    applying $(1-x)^n\leq \mathrm{e}^{-nx}$ yields ($\star$) in \eqref{eq:minimum-iid-lower}, and applying $(1-x)^n\geq 1-nx$ yields ($\ast$).
  \end{proof}
\end{lemma}
\noindent
We are ready to prove Proposition \ref{prop:weighted_distance_in_t}.
\begin{proof}[Proof of Proposition \ref{prop:weighted_distance_in_t}]
    Consider the greedy path $\pi^{\mathrm{gr}}$ starting from some $q'=q'(s_0)$
    as defined in Definition \ref{def:layers_connectors_greedy}.
    The definition of $K_t$ in \eqref{eq:kt} ensures that $\pi^{\mathrm{gr}}$
    ends in $\mathrm{Inner}_{\alpha}$. By construction of $\pi^{\mathrm{gr}}$, its
    total weight is bounded by the formula in \eqref{eq:weighted_distance_in_t_sum_tk}.
    Lemma \ref{lemma:error_prob_tk} ensures that w/p at least $1-\delta_{\ref{lemma:error_prob_tk}}^{(s_0)}$, the number of $\alpha$-connectors $(n_k)_{k\geq1}$ in \eqref{eq:small_neighbourhoods_error_outline}
    is at least
    \[
    n_k:=c_{\ref{lemma:error_prob_tk}}s_0^{c_{\ref{lemma:kt_kast}}\left(\tau-2\right)^{-(k-1)}\varepsilon_{k-1}/2}.
    \]
    By Lemma \ref{lemma:kt_kast}, $n_k\gg1$
    for all $k\geq 0$ when $s_0$ is sufficiently large.
    We bound the total weight on the greedy path using \eqref{eq:weighted_distance_in_t_sum_nk}.
    Applying ($\star$) in \eqref{eq:minimum-iid-lower} from Lemma \ref{lemma:minimum_iid_rvs}
     and a union bound over $k\in[K_t]$, obtains for some $\xi>0$ and all $s_0$ sufficiently large
    \begin{align}
      \prob\Bigg(
        \bigcup_{k\in[K_t]}
        \bigg\{
        \min_{j\in[n_k]}&\Big(L_{1j}^{(k)}+L_{2j}^{(k)}\Big)
        \geq F_{L_1+L_2}^{(-1)}
              \Big(n_k^{-(1-\xi)}
              \Big)
        \bigg\}
      \Bigg)\nonumber\\
      &\leq
      \sum_{k\in[K_t]}\exp\left(-c_{\ref{lemma:error_prob_tk}}^\xi s_{k-1}^{\xi{{\varepsilon_{k-1}/2}}}\right)
      \leq
      \sum_{k\in\mathbb{N}}\exp\left(-c_{\ref{lemma:error_prob_tk}}^\xi s_{k-1}^{\xi{{\varepsilon_{k-1}/2}}}\right)
      =:\delta_{\eqref{eq:min_error_prob}}^{(s_0)},\label{eq:min_error_prob}
    \end{align}
    which can be made arbitrarily small by choosing $s_0$ large enough, similar to the reasoning below \eqref{eq:error_prob_tk_error_prob_sum}.
    We can choose $s_0$ so large that the error probabilities
    $
    \delta_{\eqref{eq:min_error_prob}}^{(s_0)}
    +\delta_{\ref{lemma:error_prob_tk}}^{(s_0)}
    \leq \delta_{\ref{prop:weighted_distance_in_t}}$.
    As we recall the formula for $(s_k)_{k\geq 0}$ from \eqref{eq:sk}
    and the \emph{weakened} upper bound on the total weight \eqref{eq:weighted_distance_in_t_sum_nk},
    we obtain for any $q'$ such that $D_{q'}(\alpha t)\geq s_0$, with error probability
    at most $\delta_{\ref{prop:weighted_distance_in_t}}$,
    \begin{align*}
      d_L^{(t)}(q', \mathcal{L}_{K_t})\leq
      \begin{dcases}
        \sum_{k\in[K_t]}F_{L_1+L_2}^{(-1)}
              \left(
                c_{\ref{lemma:error_prob_tk}}^{-(1-\xi)}s_0^{-(1-\xi)c_{\ref{lemma:kt_kast}}\left(\tau-2\right)^{-(k-1)}(k+1)^{-2}/2}
              \right),&\text{(FV)},\\
              \sum_{k\in[K_t]}F_{L_1+L_2}^{(-1)}
                    \left(
                      c_{\ref{lemma:error_prob_tk}}^{-(1-\xi)}s_0^{-(1-\xi)c_{\ref{lemma:kt_kast}}\left((1-\varepsilon_G)/(\tau-2)\right)^{-(k-1)}\varepsilon_{G}/2}
                    \right),&\text{(G),}
      \end{dcases}
    \end{align*}
    where we annotated the lines with (FV) if it holds for FPA and VPA under condition \eqref{eq:tightness_cond},
    and with (G) otherwise. We continue to do so below.
    Using \eqref{eq:claim_sk_kt_kast}, there exists a constant $c$, such that
    \begin{equation}
    d_L^{(t)}(q', \mathcal{L}_{K_t})
    \leq
    \begin{dcases}
      \sum_{k\in[K_t]}F_{L_1+L_2}^{(-1)}
          \left(
          \exp\left(-c\log(s_0)(\tau-2)^{-k}(k+1)^{-2}\right)
          \right),&\text{(FV)} \\
          \sum_{k\in[K_t]}F_{L_1+L_2}^{(-1)}
              \left(
              \exp\left(-c\log(s_0)(1-\varepsilon_G)^k(\tau-2)^{-k}\varepsilon_G\right)
              \right),&\text{(G)}
    \end{dcases}
    \label{eq:upperbound_before_remove_convolution}
  \end{equation}
    w/p at least $1-\delta_{\ref{prop:weighted_distance_in_t}}$ if $s_0$ is sufficiently large.
    We bound the above sums, so that the terms match the terms of $Q_t$ in \eqref{eq:k_ast_qt}:
    first we remove the convolution, then we switch to an integral, apply a variable transformation, and eventually switch back to a sum.
    The sum in \eqref{eq:upperbound_before_remove_convolution} is taken over $F_{L_1+L_2}^{(-1)}$, while
    the summand in $Q_t$ is taken over $F_{L}^{(-1)}$. We relate the two inverses for $x<0$ by
    \[
    F_{L_1 + L_2}(x) = \prob(L_1 + L_2\leq x)\geq \prob\left(\max\{L_1, L_2\}\leq x/2\right) = \left(F_L\left(x/2\right)\right)^2.
    \]
    Hence, for any $z>0$, it holds that $
    F_{L_1+L_2}^{(-1)}(z) \leq 2 F_L^{(-1)}(\sqrt{z}).
    $
    Applying this to the rhs of \eqref{eq:upperbound_before_remove_convolution} obtains for a different constant $c$
    \begin{align}
      d_L^{(t)}(q', \mathcal{L}_{K_t})
      \leq
      \begin{dcases}
        2\sum_{k\in[K_t]}F_{L}^{(-1)}
            \Big(
            \exp\left(-c\log(s_0)(\tau-2)^{-k}(k+1)^{-2}\right)
            \Big), &\text{(FV),}\\
        2\sum_{k\in[K_t]}F_{L}^{(-1)}
            \Big(
            \exp\left(-c\log(s_0)(1-\varepsilon_G)^k(\tau-2)^{-k}\varepsilon_G\right)
            \Big), &\text{(G).}
      \end{dcases}
      \label{eq:upperbound_sum_before_int111}
    \end{align}
    For technical convenience, we define $a:=\inf\{x:F_L^{(-1)}(x)>0\}$ and $L' := L - a$,
    so that for (FV)
    \begin{align}
      d_L^{(t)}(q', \mathcal{L}_{K_t})
      \leq
        2aK_t +
        2\sum_{k\in[K_t]}F_{L'}^{(-1)}
            \Big(
            \exp\left(-c\log(s_0)(\tau-2)^{-k}(k+1)^{-2}\right)
            \Big).
      \label{eq:upperbound_sum_before_int}
    \end{align}
     Observe that for monotone non-increasing functions $g$, $g(1)<\infty$
     \begin{equation}
         \sum_{k=\lfloor a\rfloor}^{\lfloor b\rfloor} g(k) \overset{(\ast)}{\geq} \int_a^{b} g(x) \mathrm{d}x, \qquad
         \int_a^b g(x)\mathrm{d}x \overset{(\star)}{\geq} \sum_{k=\lceil a\rceil + 1}^{\lfloor b\rfloor}g(k).
         \label{eq:int_sum_bounds}
     \end{equation}
    We apply $(\star)$  to switch in \eqref{eq:upperbound_sum_before_int} from a sum to an integral.
    \noindent We discuss FPA and VPA under condition \eqref{eq:tightness_cond} first. We apply the variable transformation
    \begin{equation}\label{eq:upperbound_var_trans_tight2}
    \frac{y}{2} =x + \frac{\log\log s_0+\log\left(c\right)-2\log(x+1)}{|\log\left(\tau-2\right)|},
    \end{equation}
    that has a solution for all $x\geq1$ if $s_0$ is sufficiently large. Now the integrand matches the summands of $Q_t$ in \eqref{eq:k_ast_qt}.
    Differentiating both sides and rearranging terms gives an implicit formula for $\mathrm{d}x$,
    that by \eqref{eq:upperbound_var_trans_tight2} can be bounded by a function that only depends on $y$ for some $C>0$, i.e.,
    \begin{equation}\label{eq:upperbound_var_trans_tight3}
    \mathrm{d}x
    = \frac{1}{2}\left(1 + \frac{2/|\log\left(\tau-2\right)|}{x+1-2/|\log\left(\tau-2\right)|}\right)\mathrm{d}y
    \leq \frac{1}{2}\left(1+\frac{C}{y}\right).
    \end{equation}
    Thus, \eqref{eq:upperbound_sum_before_int} and \eqref{eq:upperbound_var_trans_tight3} together yield
    \begin{align*}
    d_L^{(t)}(q', \mathcal{L}_{K_t})
    &\leq
    2aK_t +
    \int_{h_\tau(s_0)}^{h_\tau(s_0)+2K_t-\frac{4}{\log(K_t+1)}{|\log\left(\tau-2\right)|}}
    \left(1+\frac{C}{y}\right)
    F_{L'}^{(-1)}
        \left(
        \exp\left(-\left(\tau-2\right)^{-y/2}\right)
        \right)\mathrm{d}y.
    \end{align*}
    where $h_\tau(s_0) = 2(\log\log s_0 + \log\left(c\right))/|\log\left(\tau-2\right)|$
    and $s_0$ is chosen so large that \eqref{eq:upperbound_var_trans_tight3} holds for all $x, y$ in the integration domain.
    When condition \eqref{eq:tightness_cond} on $L$ holds, there exists $M>0$, such that
    \begin{align*}
      \int_{h_\tau(s_0)}^{h_\tau(s_0)+2K_t}
      \frac{C}{y}
      F_{L'}^{(-1)}
          \left(
          \exp\left(-\left(\tau-2\right)^{-y/2}\right)
          \right)\mathrm{d}y
      < M.
    \end{align*}
    We apply $(\ast)$ in \eqref{eq:int_sum_bounds} to switch back to a sum. As this sum contains at most $2K_t + 1$ terms, there exists a larger $M$, recalling that $L=L'+a$, so that
    \begin{align*}
    d_L^{(t)}(q', \mathcal{L}_{K_t})
    &\leq
    2aK_t +
    \sum_{k=\lfloor h_{\tau}(s_0)\rfloor}^{\lfloor h_\tau(s_0)+2K_t\rfloor}
    F_{L'}^{(-1)}
        \left(
        \exp\left(-\left(\tau-2\right)^{-y/2}\right)
        \right)
    +M \\
    &\leq
    \sum_{k=\lfloor h_{\tau}(s_0)\rfloor}^{\lfloor h_\tau(s_0)+2K_t\rfloor}
    F_{L}^{(-1)}
        \left(
        \exp\left(-\left(\tau-2\right)^{-k/2}\right)
        \right)
    +M.
    \end{align*}
    Application of the bound \eqref{eq:claim_sk_kt_kast} on $2K_t$ yields the assertion \eqref{eq:weighted_distance_in_t}
    for FPA and VPA under \eqref{eq:tightness_cond}.

    For FPA and VPA if \eqref{eq:tightness_cond} does not hold, and for the model GVPA, we use a similar variable transformation to \eqref{eq:upperbound_var_trans_tight3} for the integral in \eqref{eq:upperbound_sum_before_int}, i.e.,
    \begin{equation}\nonumber%\label{eq:upperbound_var_trans_nontight1}
    \frac{y}{2}\frac{|\log\left(\tau-2\right)|}{\log\left((1-\varepsilon_G)/(\tau-2)\right)}
    = x + \frac{\log\log s_0+\log\left(c\varepsilon_G\right)}{\log\left((1-\varepsilon_G)/(\tau-2)\right)},
    \end{equation}
    which yields combined with the second line in \eqref{eq:upperbound_sum_before_int111}
    \[
    d_L^{(t)}(q', \mathcal{L}_{K_t})
    \leq
    (1+\varepsilon_{\ref{prop:weighted_distance_in_t}})
    \int_{\tilde{h}_{\tau}(s_0)}^{\tilde{h}_\tau(s_0) + \frac{2}{1+\varepsilon_{\ref{prop:weighted_distance_in_t}}}K_t}
    F_{L}^{(-1)}
          \left(
          \exp\left(-\left(\tau-2\right)^{-x}\right)
          \right)\mathrm{d}x,
    \]
    where
    $
    \tilde{h}_\tau(s_0) =  2\left(\log\log s_0+\log\left(c\varepsilon_G\right)\right)/|\log\left(\tau-2\right)|.
    $
    After applying the bound \eqref{eq:claim_sk_kt_kast_general} on $K_t$ and switching back to a sum using $(\ast)$ from \eqref{eq:int_sum_bounds},
    we obtain the desired bound \eqref{eq:weighted_distance_in_t_general}.
\end{proof}
 This establishes all prerequisites. We turn to the  upper bound for the conservative case.
\subsection{Conservative case}
\begin{proof}[Proof of Proposition \ref{prop:upper_bound}]
  We first show the result for FPA and VPA under \eqref{eq:tightness_cond}.
  At the end of this proof, we argue how to adapt the proof for GVPA and VPA or FPA if \eqref{eq:tightness_cond} does not hold.
  Fix  $\delta>0$, and let $\delta_{\ref{prop:graph_distance_to_high_degree}}=\delta_{\ref{prop:weighted_distance_in_t}}=\delta/8$,
  $\alpha=1-\delta/16$.
We can choose $t$ sufficiently large,
such that
\begin{enumerate}[label=(\roman*),leftmargin=2\parindent]
  \item by Proposition \ref{prop:graph_distance_to_high_degree}
  for $s_{\ref{prop:graph_distance_to_high_degree}}=s_0$ that we choose below in (ii),
  there exists a constant $C_{\ref{prop:graph_distance_to_high_degree}}=C_{\ref{prop:graph_distance_to_high_degree}}(\delta, s_0)$,
  such that for $q\in\{u,v\}$ there is a vertex $q'\in[\alpha t]$ with degree at least $s_0$
  within graph distance $C_{\ref{prop:graph_distance_to_high_degree}}$ w/p at least $1-\delta/8$.
  \item by Proposition \ref{prop:weighted_distance_in_t}
there is an $s_0=s_0(\delta, \varepsilon_{\ref{prop:weighted_distance_in_t}})>0$, such that for  $q'\in[\alpha t]$ with $D_{\alpha t}(q')\geq s_0$, the weighted distance to
the inner core is not too large.
As the terms in the sum in \eqref{eq:weighted_distance_in_t}  are decreasing, we shift the summation bounds to match the bounds from $Q_t$ in \eqref{eq:k_ast_qt}, so that
\begin{equation}
  \prob\Big(d_L^{(t)}(q',\mathrm{Inner}_{\alpha})\geq M_{\ref{prop:weighted_distance_in_t}}+Q_t\Big) \leq \frac{\delta}{8}.
  \label{eq:prop_proof_distance_to_inner_t}
\end{equation}
\item by Proposition \ref{prop:innercore_bounded} for $\delta_{\ref{prop:innercore_bounded}}=\delta/8$,
the graph distance between $w_1, w_2\in \mathrm{Inner}_{\alpha}$  is smaller than $C_{\ref{prop:innercore_bounded}}$
w/p at least $1-\delta/8$.
\item for $(L_j)_{j\geq0}$ i.i.d.\ copies of $L$,
the sum of constantly many weights is negligible compared to the diverging sequence $Q_{t}$ defined in \eqref{eq:k_ast_qt}, i.e.,
there exists $M'>0$ such that
\begin{equation}
\prob\bigg(\sum_{j\in[2C_{\ref{prop:graph_distance_to_high_degree}}+C_{\ref{prop:innercore_bounded}}]} L_j\geq M'\bigg)\leq \frac{\delta}{8}.
\label{eq:prop_proof_finite_sum_bounded}
\end{equation}
\end{enumerate}
Conditionally on the intersection of the complements of the events in
\eqref{eq:prop_proof_distance_to_inner_t} and (i) for $q\in\{u,v\}$,
we can construct greedy paths from $u$ and $v$ to the inner core. Hence,
\begin{equation}
\prob\bigg(\bigcup_{q\in\{u,v\}}\Big\{d_L^{(t)}(q,\mathrm{Inner}_\alpha)\geq
Q_{t} + M_{\ref{prop:weighted_distance_in_t}} + \sum_{j\in[C_{\ref{prop:graph_distance_to_high_degree}}]} L_j^{(q)} \ \Big \vert\  q<\alpha t\Big\}\bigg)
\leq\frac{3\delta}{4},
\label{eq:distance_q_to_inner}
\end{equation}
where the last sum between brackets in \eqref{eq:distance_q_to_inner} represents the weighted distance
from $q$ to some vertex with degree at least $s_0$.
By (iii), we can bound the distance between two vertices in the inner core,
and thus we have constructed a path from $u$ to $v$.
Hence, by a union bound over the events in (iii) and \eqref{eq:distance_q_to_inner},
we can bound the total weight on this path, i.e.,
\begin{equation}
\prob\bigg(d_L^{(t)}(u,v)\geq
2Q_{t}  + \sum_{j\in[2C_{\ref{prop:graph_distance_to_high_degree}}+C_{\ref{prop:innercore_bounded}}]} L_j \ \Big \vert \  u,v<\alpha t\bigg)
\leq\frac{7\delta}{8}.
\label{eq:prop_proof_qt_and_finite_sum}
\end{equation}
Combining \eqref{eq:prop_proof_finite_sum_bounded} with \eqref{eq:prop_proof_qt_and_finite_sum}
yields by our choice of $\alpha = 1-\delta/16$ that
\begin{equation}
\prob\left(d_L^{(t)}(u,v)\geq 2Q_{t} + 2M_{\ref{prop:weighted_distance_in_t}} + M'\right)\leq \delta,
\label{eq:prop_proof_qt_result}
\end{equation}
finishing the proof for FPA and VPA under \eqref{eq:tightness_cond}.
For the more general case, only \eqref{eq:prop_proof_distance_to_inner_t} does not hold, which
can be replaced by
\begin{equation*}
  \prob\Big(d_L^{(t)}(q',\mathrm{Inner}_{\alpha})\geq (1+\varepsilon_{\ref{prop:weighted_distance_in_t}})Q_t\Big) \leq \frac{\delta}{8}.
\end{equation*}
Propagating the rhs between brackets through \eqref{eq:distance_q_to_inner} and  \eqref{eq:prop_proof_qt_result} obtains the result for GVPA.
\end{proof}
\subsection{Explosive case}
Similarly to the conservative case we create a path from $u$ to $v$ and use the total weight on the path as an upper bound
for the weighted distance. Again, by applying Propositions \ref{prop:graph_distance_to_high_degree}, \ref{prop:weighted_distance_in_t}, and \ref{prop:innercore_bounded},
this path goes from $q\in\{u,v\}$ to a vertex $q'$ with degree at least $s_0$, after which we connect this vertex to the inner core. The total weight on the segments to the inner core can be made arbitrarily small for large $t$ and $s_0$.
However, it is not straightforward to see that the weight on the first parts of the paths converge to the explosion times of two $\mathrm{LWLs}$, as we increase the degree $s_0$.
We start by showing that the explosion time of the LWL is finite, for which we use Propositions \ref{prop:graph_distance_to_high_degree} and \ref{prop:weighted_distance_in_t}. Recall the formal statement of Theorem \ref{theorem:explosion_time_ppg}.
\begin{proof}[Proof of Theorem \ref{theorem:explosion_time_ppg}]
Fix $\delta>0$. Recall $K_t$ from \eqref{eq:kt} and $C_{\ref{prop:graph_distance_to_high_degree}}$ for $\delta_{\ref{prop:graph_distance_to_high_degree}}=\delta/4$
from Proposition \ref{prop:graph_distance_to_high_degree}.
Let $a(t) := \min\{\kappa(t), 2K_t+C_{\ref{prop:graph_distance_to_high_degree}}\}$,
where $\kappa(t)$ denotes the maximum number of generations in the LWL to maintain a coupling with $\mathrm{PA}_t$
w/p at least $1-\delta/4$, from \eqref{eq:local_weak_limit}.
As $\kappa(t)$ and $K_t$ tend to infinity with $t$, the same holds for $a(t)$.
Denote the first $k$ generations of the LWL rooted in $\circledcirc$ by $\mathrm{LWL}_k(\circledcirc)$. Define
\begin{align*}
  X_t &:= d_L^{(t)}\left(q^{(t)}, \partial \mathcal{B}_G^{(t)}(q^{(t)}, 2K_t +C_{\ref{prop:graph_distance_to_high_degree}})\right), \\
  Y_t &:= d_L^{(t)}\left(q^{(t)}, \partial \mathcal{B}_G^{(t)}(q^{(t)}, a(t))\right)\mathds{1}_{\{\text{coupling succesful}\}}=\beta_{a(t)}^{\mathrm{LWL}_{a(t)}\left(q^{(t)}\right)}\mathds{1}_{\{\text{coupling succesful}\}}.
\end{align*}
Observe that $X_t$ can be viewed as an upper bound of the weighted distance to $\mathrm{Inner}_\alpha$.
By the choice of $a(t)$, $X_t$ stochastically dominates $Y_t$, i.e., $\prob(Y_t \leq X_t)=1$.
Consider the subsequence of times, defined recursively as
\begin{equation}
  t_i := \begin{dcases}
  1, &\, i=0,\\
  \min\{t: a(t) > a(t_{i-1}) \text{ and } K^\ast_{t}>K^\ast_{t_{i-1}}\},&\, i>0.
\end{dcases}\nonumber\end{equation}
By construction of $Y_t$ and $(t_i)_{i\geq0}$,
$
\prob(\beta_{\infty}^{\mathrm{LWL}}\leq y) \,=\, \lim_{i\rightarrow\infty} \prob(Y_{t_i}\leq y).
$
Since $X_{t_i}$ dominates $Y_{t_i}$,
\begin{align}
\prob(\beta_{\infty}^{\mathrm{LWL}}=\infty)
=
\lim_{M\rightarrow\infty}\prob(\beta_{\infty}^{\mathrm{LWL}}\geq M)
=
\lim_{M\rightarrow\infty}\lim_{i\rightarrow\infty}\prob(Y_{t_i}\geq M)
\leq
\lim_{M\rightarrow\infty}\lim_{i\rightarrow\infty}\prob(X_{t_i}\geq M).
\label{eq:y_infty_bound_by_xt}
\end{align}
Below we find a bound on $\prob(X_{t_i}\geq M)$ that does not depend on $t$ or $i$ to obtain the result.
Recall $Q_t$ from \eqref{eq:k_ast_qt} and define
$
Q_{\infty}:= \lim_{t\rightarrow\infty} Q_t,
$
which is finite since $I(L)<\infty$.
By a union bound on the events in Proposition \ref{prop:graph_distance_to_high_degree}
and Proposition \ref{prop:weighted_distance_in_t}, for any $\delta, \varepsilon>0$, if $i$ is sufficiently large,
\[
\prob\bigg(X_{t_i} \geq (1+\varepsilon)Q_{t_i} + \sum_{j\in[C_{\ref{prop:graph_distance_to_high_degree}}]} L_j\bigg)\leq \frac{\delta}{2}.
\]
Note that $(1+\varepsilon)Q_{t_i} \leq\ (1+\varepsilon)Q_{\infty}$.
Choose $M'=M'(\delta)$ so that $M'/2\geq(1+\varepsilon)Q_\infty$ and $\sum_{j\in[C_{\ref{prop:graph_distance_to_high_degree}}]} L_j\geq M'/2$ w/p at most $\delta/2$, similarly to \eqref{eq:prop_proof_finite_sum_bounded}.
Hence,
\[
\lim_{i\rightarrow\infty}\prob(X_{t_i}\geq M') \leq \delta.
\]
Recall \eqref{eq:y_infty_bound_by_xt} to obtain
$
\prob(\beta_{\infty}^{\mathrm{LWL}}=\infty)\leq \delta.
$
Since $\delta>0$ was arbitrary, $\beta_{\infty}^{\mathrm{LWL}}<\infty$ almost surely.
\end{proof}
Recall the definitions of the graph neighbourhoods in Definition \ref{def:graph_distances_1} and the explosion time in Definition \ref{def:explosive_graph}.
Using the finiteness of the explosion time, we bound the weight on a path to a vertex with degree at least $s_0$.
To do so, we need the following general lemma.
\begin{lemma}[Reaching a high-degree vertex in an explosive tree]\label{lemma:no_explosion_bounded_degrees}
  Consider a (possibly random) locally finite tree $T$, rooted in $\circledcirc$, where every edge is equipped with an i.i.d.\ edge-weight from distribution $L$,
  where $F_L(0)=0$, such that $T$ is explosive. Write $D(x)$ for the degree of vertex $x$.
  Fix $\delta_{\ref{lemma:no_explosion_bounded_degrees}}>0$. For any $s\in\mathbb{N}$, there exists $N=N(\delta_{\ref{lemma:no_explosion_bounded_degrees}}, s)<\infty$ such that
  \begin{equation}
  \prob\left(\mathcal{B}_L^T\left(\circledcirc, \sigma_N^T\right)\cap \{x\in T: D(x)\geq s\}\neq \varnothing\right) \,\geq\, 1-\delta_{\ref{lemma:no_explosion_bounded_degrees}}.
  \label{eq:no_explosion_bounded_degrees}
  \end{equation}
  \begin{proof}
    The event in \eqref{eq:no_explosion_bounded_degrees} means that among the $N$ closest vertices to $\circledcirc$, there is a vertex with degree at least $s$.
    We argue by contradiction.
    If the tree $T$ is explosive and the vertices contained in $\bigcup_{n\in\mathbb{N}}\mathcal{B}_L(\circledcirc, \sigma_n)$ would all have degree at most $s$,
    then the forest restricted to vertices in $T$ with degree at most $s$ is also explosive. This forest consists of trees with at most exponentially growing generation sizes. Hence, \cite[Lemma 4.3]{komjathy2018explosion} applies and explosion is impossible,
    \noindent i.e.,
    \[
    \prob\bigg(\bigcap_{n\geq1}\left\{\mathcal{B}_L^T\left(\circledcirc, \sigma_n^T\right)\cap \{x\in T: D(x)\geq s\}= \varnothing\right\}\bigg) = 0.
    \]
    From here we obtain the result since  $\bigcup_{n\in[N]}\mathcal{B}_L^T\left(\circledcirc, \sigma_n^T\right)=\mathcal{B}_L^T\left(\circledcirc, \sigma_N^T\right)$:
    \begin{align*}
      \lim_{N\rightarrow\infty} &\prob\bigg(\left\{\mathcal{B}_L^T\left(\circledcirc, \sigma_N^T\right)\cap \{x\in T: D(x)\geq s\}\neq \varnothing\right\}\bigg) \\
      &= \lim_{N\rightarrow\infty} \prob\bigg(\bigcup_{n\in[N]}\left\{\mathcal{B}_L^T\left(\circledcirc, \sigma_n^T\right)\cap \{x\in T: D(x)\geq s\}\neq \varnothing\right\}\bigg) = 1,
    \end{align*}
    hence, the lhs will have probability at least $1-\delta_{\ref{lemma:no_explosion_bounded_degrees}}$ for a sufficiently large $N$.
  \end{proof}
\end{lemma}
In the next lemma, we exploit the coupling to LWL to find a vertex with sufficiently high degree.
This lemma can be viewed as an extension of Proposition \ref{prop:graph_distance_to_high_degree}.
\begin{lemma}[Weighted distance to a vertex with degree at least $s$]\label{lemma:weighted_distance_to_high_degree}
  Consider PA under the same conditions as Theorem \ref{thm:weighted_distance_explosive} at time $t$.
Let $q$ be a typical vertex. For any $\delta_{\ref{lemma:weighted_distance_to_high_degree}}, s_0>0$, there exists $N_{\ref{lemma:weighted_distance_to_high_degree}}\in\mathbb{N}$, such that when $t$ is sufficiently large
\[
\prob\bigg(\mathcal{B}_L^{(t)}\left(q, \sigma_{N_{\ref{lemma:weighted_distance_to_high_degree}}}^{(t)}\right)\cap \{x\in[t]:D_t(x)\geq s_0\}= \varnothing\bigg) \leq \delta_{\ref{lemma:weighted_distance_to_high_degree}}.
\]
\begin{proof}
  From Lemma \ref{lemma:no_explosion_bounded_degrees} for $\delta_{\ref{lemma:no_explosion_bounded_degrees}}=\delta/2$, we obtain that in the limiting object $\mathrm{LWL}(\circledcirc)$ for a root $\circledcirc$,
  there is an $N$ such that a vertex of degree at least $s_0$ is reached before time $\sigma_N$ w/p at least $1-\delta/2$. By Theorem \ref{theorem:explosion_time_ppg} the LWL is explosive.
  By applying Proposition \ref{prop:local_weak_limit}, we can assume
  $t$ is so large that we can maintain a coupling between w/p at least $1-\delta/2$ up to graph distance $N$,
  so that $\widetilde{\mathcal{B}}^{(t)}_G(q, N)\simeq \mathrm{LWL}_N(\circledcirc)$.
  As the $N$-th closest vertex in $L$-distance is within the $G$-neighbourhood of size $N$, we obtain the result.
\end{proof}
\end{lemma}
\noindent
We are now ready to prove the upper bound for the explosive case.
\begin{proof}[Proof of Proposition \ref{prop:upperbound_explosive}]
  Let $s_0$ be the constant we obtain from Proposition \ref{prop:weighted_distance_in_t}
  when setting
  $ \delta_{\ref{prop:weighted_distance_in_t}}=\delta/8$, $\alpha_{\ref{prop:weighted_distance_in_t}}=1-\delta/16$ and
  $\varepsilon_{\ref{prop:weighted_distance_in_t}}=\varepsilon/3$,
  such that the sum in \eqref{eq:weighted_distance_in_t} is smaller than $\varepsilon/6$.
  Abbreviate $t':=\lfloor\alpha t\rfloor$.
  Observe that
  \begin{equation}
  \prob\left(\{u\notin[t']\}\cup\{v\notin[t']\}\right)\leq \delta / 8.
  \label{eq:event_uv_young_prob}
  \end{equation}
  We call the above event between brackets $\mathcal{E}_1$. Recall $\sigma_n$ from \eqref{eq:sigma_n}.
  Let for $q\in\{u,v\}$
  \begin{equation*}
    N^\ast(q)=\min\{n: \exists x\in \mathcal{B}_L^{(t')}(q,\sigma_n^{(t')}):D_{t'}(x)\geq s_0\}.
  \end{equation*}
  Let $u'$ be the vertex in $\mathcal{B}_L^{(t')}\big(u,\sigma_{N^\ast}^{(t')}(u)\big)$ that has degree at least $s_0$,
  and $v'$ analogously.
  The key observation is that, conditionally on $\mathcal{E}_1^c$, we can use a sprinkling argument and look at the graph at two moments in time, i.e.,
  \[
  d_L^{(t)}(u,v) \leq d_L^{(t')}(u, u') + d_L^{(t')}(v, v') + d_L^{(t)}(u', v').
  \]
  The above is true, since for \emph{fixed} vertices the weighted distance between them is non-increasing in time.
  Conditionally on the event $\mathcal{E}_1^c$, $u$ and $v$ are uniform vertices in $[t']$,
  so that we can still apply Lemma \ref{lemma:weighted_distance_to_high_degree} for $s_0$ and $\delta_{\ref{lemma:weighted_distance_to_high_degree}}=\delta/8$,
  i.e., for $t$ sufficiently large and some $N_{\ref{lemma:weighted_distance_to_high_degree}}=N_{\ref{lemma:weighted_distance_to_high_degree}}(\delta, s_0)$,
  \begin{equation}
  \prob\bigg(\bigcup_{q\in\{u,v\}}\left\{N^\ast(q) \geq N_{\ref{lemma:weighted_distance_to_high_degree}}\right\}\bigg)\leq \frac{\delta}{4}.
  \label{eq:event_q_close_large_degree}
  \end{equation}
  We call the  above event between the $\prob$-sign $\mathcal{E}_2$.
  By the choice of $u'$ and $v'$, we obtain that
  \begin{align*}
  \prob\left(d_L^{(t)}(u,v)
  \leq
  \sigma_{N_{\ref{lemma:weighted_distance_to_high_degree}}}^{(t')}(u) + \sigma_{N_{\ref{lemma:weighted_distance_to_high_degree}}}^{(t')}(v)+ d_L^{(t)}(u', v') \mid \mathcal{E}_1^c\cap \mathcal{E}_2^c\right) = 1.
  \end{align*}
  Assume $t'$ is so large that we can maintain a coupling with LWL w/p at least $1-\delta/8$ for $q\in\{u,v\}$ up to generation $N_{\ref{lemma:weighted_distance_to_high_degree}}$,
  possible by Propositions \ref{prop:local_weak_limit}.
  If the coupling is successful, then
  $\sigma_{N_{\ref{lemma:weighted_distance_to_high_degree}}}^{(t')}(q)\leq\beta_{N_{\ref{lemma:weighted_distance_to_high_degree}}}^{\mathrm{LWL}^{(q)}},$
  for $q\in\{u,v\}$, where the two random variables are independent copies of $\beta^{\mathrm{LWL}}_{N_{\ref{lemma:weighted_distance_to_high_degree}}}$. Thus, combining the coupling error and the error probabilities on $\mathcal{E}_1$ and $\mathcal{E}_2$ in \eqref{eq:event_uv_young_prob} and \eqref{eq:event_q_close_large_degree},
  \begin{equation}
    \prob\left(d_L^{(t)}(u,v)\leq\beta_{N_{\ref{lemma:weighted_distance_to_high_degree}}}^{\mathrm{LWL}^{(u)}} + \beta_{N_{\ref{lemma:weighted_distance_to_high_degree}}}^{\mathrm{LWL}^{(v)}} + d_L^{(t)}(u', v')\right)
    \geq 1-5\delta/8,
    \label{eq:theorem_explosive_upperbound_distance_to_ppg_1}
  \end{equation}
  where $\prob$ denotes the probability measure on the coupled probability space.
  Recall \eqref{eq:weighted_distance_explosive_upperbound}, we observe that we are left to prove
  \begin{equation}
  \prob\left(d_L^{(t)}(u',v')\geq \varepsilon\right) \leq 3\delta/8.
  \label{eq:theorem_explosive_upperbound_uv_large_degree}
  \end{equation}
  By our choice of $\delta_{\ref{prop:weighted_distance_in_t}}=\delta/8$ and $\varepsilon_{\ref{prop:weighted_distance_in_t}}=\varepsilon/3$
  above, we obtain by Proposition \ref{prop:weighted_distance_in_t}
  \begin{equation}\label{eq:upper-explosive-bound-q-prime}
  \prob\bigg(\bigcup_{q'\in\{u',v'\}}\big\{d_L^{(t)}(q', \mathrm{Inner}_{\alpha}) \geq \varepsilon/3 \big\}\bigg) \leq \delta/4.
  \end{equation}
  It is important to note that although $u'$ and $v'$ are \emph{special} vertices at time $\alpha t$, \eqref{eq:upper-explosive-bound-q-prime} holds independently of $N_{\ref{lemma:weighted_distance_to_high_degree}}$, as the path to inner core, constructed in the proof of Proposition \ref{prop:weighted_distance_in_t},  uses only edges that arrived after $\alpha t$.
To apply Proposition \ref{prop:innercore_bounded}, we let $\delta_{\ref{prop:innercore_bounded}}=\delta/8$ and $\varepsilon_{\ref{prop:innercore_bounded}}=\varepsilon/3$.
  By a union bound over the events in Propositions \ref{prop:weighted_distance_in_t} and \ref{prop:innercore_bounded}
  there is a path from $u'$ to $v'$ of weight at most $\varepsilon/3$, yielding \eqref{eq:theorem_explosive_upperbound_uv_large_degree}.
  Combining \eqref{eq:theorem_explosive_upperbound_distance_to_ppg_1} and \eqref{eq:theorem_explosive_upperbound_uv_large_degree}
  obtains the desired bound \eqref{eq:weighted_distance_explosive_upperbound}.
\end{proof}

\section{Lower bound on the weighted distance}\label{sec:lowerbound}
The next propositions state the lower bounds for Theorem \ref{thm:weighted_distance_nonexplosive} and \ref{thm:weighted_distance_explosive},
which are the counterparts of Propositions \ref{prop:upper_bound} and \ref{prop:upperbound_explosive}.
The lower bound of Theorem \ref{theorem:conservative_lwl} which follows from the same proof techniques, is postponed to Section \ref{sec:cons-lwl}.
\begin{proposition}[Lower bound on the weighted distance, conservative case]\label{prop:lowerbound_unexplosive}
  Consider PA under the same conditions as Theorem \ref{thm:weighted_distance_nonexplosive}.
  Recall $I(L) = \infty$.
  Let $u, v$ be two typical vertices. Then for
  every $\delta, \varepsilon > 0$, if $t$ is sufficiently large,
  \begin{equation}
    \prob\left(d_L^{(t)}(u,v) \geq (1-\varepsilon)2Q_t\right)
    \geq 1-\delta.\label{eq:weighted_distance_nonexplosive}
  \end{equation}
  Moreover, for the models FPA and VPA from Definition \ref{def:pa_m_delta} and \ref{def:pa_gamma},
  there exists a constant $M_{\ref{prop:lowerbound_unexplosive}}=M_{\ref{prop:lowerbound_unexplosive}}(\delta)$ such that for $t$ sufficiently large
  \begin{equation}
  \prob\left(d_L^{(t)}(u,v) \geq 2Q_t - 2M_{\ref{prop:lowerbound_unexplosive}}\right) \geq 1-\delta.
  \label{eq:weighted_distance_nonexplosive_tight}
  \end{equation}
\end{proposition}
\noindent Observe that Proposition \ref{prop:lowerbound_unexplosive} gives a tight lower bound on the weighted distance for any distribution, i.e., it does not depend on the tightness condition \eqref{eq:tightness_cond}. Thus, only the upper bound in \eqref{eq:weighted_distance_nonexplosive_upperbound} needs be improved to prove \eqref{eq:weighted-distance-tight} regardless of \eqref{eq:tightness_cond}.

We proceed to the main proposition for the explosive edge-weight distributions.
\begin{proposition}[Lower bound on the weighted distance, explosive case]\label{prop:lowerbound_explosive}
  Consider  PA under the same conditions as Theorem \ref{thm:weighted_distance_explosive}.
    Recall $I(L) < \infty$.
  Let $u, v$ be two typical vertices. Then, there is a coupled probability space,  such that for any $\delta>0$, there exists
  some function $a(t)$ that tends to infinity with $t$, such that
  \begin{equation}
    \prob\bigg(d_L^{(t)}(u,v)\geq \beta^{\mathrm{LWL}_{a(t)}^{(u)}}_{a(t)} + \beta^{\mathrm{LWL}_{a(t)}^{(v)}}_{a(t)}\bigg)
    \geq 1-\delta,\nonumber%\label{eq:weighted_distance_explosive_lowerbound}
  \end{equation}
  where $\beta^{\mathrm{LWL}_{a(t)}^{(u)}}_{a(t)}$ and $\beta^{\mathrm{LWL}_{a(t)}^{(u)}}_{a(t)}$
  are the times to reach graph distance $a(t)$ in the LWLs coupled to $u$, $v$, respectively.
\end{proposition}
\noindent As the proofs of the propositions are partly based on the same principles, we outline the proofs simultaneously,
after which we prove the required lemmas, and then give a separate proof for the above propositions. For notational convenience, we only outline the models FPA and VPA.
\subsection*{Outline of the proofs}
Let $u, v$ be two typical vertices.
Recall $K^\ast_t$ from \eqref{eq:k_ast_qt}.
From \cite[Theorem 2]{dereich2012typical} it follows that
any path connecting $u$ and $v$ has whp at least $2K^\ast_t-M'_t$ edges w/p close to 1, for some function $M'_t=o(K^\ast_t)$.
The idea of the proof of Proposition \ref{prop:lowerbound_unexplosive} is to show that any path
starting from $u$ or $v$ that has $K^\ast_t-M'_t$ edges, has total weight at least $Q_t-M_{\ref{prop:lowerbound_unexplosive}}$, or total weight at least $\beta^{\mathrm{LWL}}_{a(t)}$, for some function $a(t)$ that tends to infinity, for the conservative and explosive case respectively.
The main steps that make this succeed,
which we formalize further below, are as follows, where (I-III) apply for the conservative case, while the explosive case follows from (I).
\begin{enumerate}[label=(\Roman*),leftmargin=2\parindent]
  \item We use results from \cite{dereich2012typical, dommers2010diameters} to conclude  that
  the neighbourhoods $\mathcal{B}_G^{(t)}(u,k)$, $\mathcal{B}^{(t)}_G(v,k)$ are disjoint whp for $k<K^\ast_t-M'_t$.
  These neighbourhoods up to graph distance $a(t)$ are coupled to two independent LWLs, from which the lower bound for the explosive case will follow.
  \item We show that $\partial\mathcal{B}_G^{(t)}(u,k)$ has whp at most $\exp\left(B(\tau-2)^{-k/2}\right)$ vertices,
  for $B$ sufficiently large and $k<K^\ast_t - M'_t$, by counting the number of paths of length $k$ starting from $q\in\{u,v\}$. The number of such paths is an upper bound for the number of vertices in $\partial\mathcal{B}_G^{(t)}(q,k)$.
\item From (I) it follows that $d_L^{(t)}(u, v)$ is at least the weighted distance from $u$
to $\partial \mathcal{B}_{G}^{(t)}(u,K^\ast_t-M'_t)$ plus the weighted distance from $v$ to $\partial \mathcal{B}_{G}^{(t)}(v, K^\ast_t-M'_t)$.
Along the same reasoning as shown by Adriaans and Komj\'athy \cite[Lower bound (2.11)]{adriaans2017weighted},
we obtain the lower bound
\begin{align}
  d_L^{(t)}(u,v)
  &\geq d_L^{(t)}\left(u, \partial \mathcal{B}_{G}^{(t)}(u, K^\ast_t-M'_t)\right)
   + d_L^{(t)}\left(v, \partial \mathcal{B}_{G}^{(t)}(v, K^\ast_t-M'_t)\right) \nonumber\\
  &\geq \sum_{q\in \{u,v\}}\sum_{i=0}^{\lfloor K^\ast_t-M'_t\rfloor-1} \min_{x\in\partial \mathcal{B}_G^{(t)}(q, i), y\in\partial \mathcal{B}_{G}^{(t)}(q, i+1)} L_{(x,y)}. \label{eq:weighted_distance_lb_sum_min}
\end{align}
The weight $L_{(x,y)}$ is the weight attached to the edge $(x,y)$ in the graph
if it is present. Otherwise, it is a new i.i.d.\ copy of $L$. Note that this is
a valid lower bound, as the minimum is non-increasing when adding more edges.
Using the bounds established in (II),
we show that this sum of minima is at least $2Q_t-2M_{\ref{prop:lowerbound_unexplosive}}$.
\end{enumerate}
Before proving Proposition \ref{prop:lowerbound_unexplosive} and \ref{prop:lowerbound_explosive},
we formally introduce the lemmas and proposition that correspond to Step (I-III).
We start with a proposition from \cite{dereich2012typical} that implies the first part of (I).
\begin{proposition}[{Typical graph distance \cite[Theorem 2]{dereich2012typical}}]\label{prop:upperbound_dg}
  Consider PA with power-law exponent $\tau\in(2, 3)$.
  Let $u, v$ be two typical vertices.
  For any $\delta_{\ref{prop:upperbound_dg}}>0$, there exists a function $M_{\ref{prop:upperbound_dg}}(t)$ such that if $t$ is sufficiently large, then
  \begin{equation}
    \prob\left(\left\{d_G^{(t)}(u,v)>2K^\ast_t - 2M_{\ref{prop:upperbound_dg}}(t)\right\}\right) \geq 1-\delta_{\ref{prop:upperbound_dg}}\label{eq:upperbound_dg},
  \end{equation}
  where the function $M_{\ref{prop:upperbound_dg}}(t)$ is of order
  $O(1)$ for the models FPA and VPA, and
    $o(K^\ast_t)$ for the model GVPA.
  We denote the above event between brackets by $\mathcal{E}_{\ref{prop:upperbound_dg}}^{(t)}$.
\end{proposition}
\noindent Indeed, as a result of Proposition \ref{prop:upperbound_dg}, on the event $\mathcal{E}_{\ref{prop:upperbound_dg}}^{(t)}$
\begin{equation*}
  \mathcal{B}^{(t)}_G(u,k) \cap \mathcal{B}^{(t)}_G(v,k) = \varnothing\qquad\text{for $k\leq K^\ast_t - M_{\ref{prop:upperbound_dg}}(t)$},
\end{equation*}
establishing (I).
We continue with the lemmas implying (II) and (III) from the outline and prove the conservative case first.
We make use of some results from Dereich, M\"{o}nch, and M\"{o}rters
\cite[Theorem 2]{dereich2012typical} where they prove lower bounds on graph
distances for random graphs. They work under the following general condition and prove that it holds for FPA and GVPA.
\begin{proposition}[PA($\gamma$) {\cite[Proposition 3.1, 3.2]{dereich2012typical}}]\label{prop:pa_g}
  We say that PA satisfies the condition PA($\gamma_{\ref{prop:pa_g}}$)
  for some $\gamma_{\ref{prop:pa_g}}\in(0, 1)$, if there exists a constant $\nu_{\ref{prop:pa_g}} > 0$, such that
  for all $t$ and pairwise distinct vertices $v_0, ..., v_l\in [t]$
  \begin{equation}\label{eq:definition_p_v0_v1}
    \prob(v_0\leftrightarrow v_1 \leftrightarrow v_2 \leftrightarrow \dots
    \leftrightarrow v_\ell) \leq \prod_{k=1}^\ell p(v_{k-1}, v_k)
    =: p(v_0, ..., v_\ell)
  \end{equation}
  where $p(m, n):=\nu_{\ref{prop:pa_g}} (m \wedge n)^{-\gamma_{\ref{prop:pa_g}}}(m\vee n)^{\gamma_{\ref{prop:pa_g}-1}}$.
  The above condition is satisfied for FPA or VPA, with $\gamma_{\ref{prop:pa_g}}=1/(\tau-1)$; and for GVPA for any $\gamma_{\ref{prop:pa_g}}>\gamma_f$.
\end{proposition}
\noindent Step (II) in the outline follows from the next lemma.
\begin{lemma}[Small probability of too large neighbourhoods]\label{lemma:small_neighbourhoods}
Consider PA under the same conditions as Proposition \ref{prop:lowerbound_unexplosive}, that satisfies Proposition \ref{prop:pa_g} for some $\gamma_{\ref{prop:pa_g}}\in(1/2, 1)$.
Let $q$ be a typical vertex.
There exist a constant $B_{\ref{lemma:small_neighbourhoods}}=B_{\ref{lemma:small_neighbourhoods}}(\gamma_{\ref{prop:pa_g}},\nu_{\ref{prop:pa_g}})>0$,
such that for any $\delta_{\ref{lemma:small_neighbourhoods}}>0$, there exists
a sequence of events $\mathcal{E}_{\ref{lemma:small_neighbourhoods}}^{(t)}(q)$,
such that for all $t$ sufficiently large
\begin{equation}
  \prob\left(\mathcal{E}_{\ref{lemma:small_neighbourhoods}}^{(t)}(q)\right)\geq1-\delta_{\ref{lemma:small_neighbourhoods}}, \label{eq:small_neighbourhoods_condition}
\end{equation}
and for any $B>B_{\ref{lemma:small_neighbourhoods}}$, $\gamma\in[\gamma_{\ref{prop:pa_g}},1)$ and $M_{\ref{prop:upperbound_dg}}(t)$ from Lemma \ref{prop:upperbound_dg} with parameter $\delta_{\ref{prop:upperbound_dg}}=\delta_{\ref{lemma:small_neighbourhoods}}$,
\begin{equation}
\prob\Bigg(
\bigcap_{k\in[K^\ast_t -M_{\ref{prop:upperbound_dg}}(t)]}
\left\{|\partial \mathcal{B}^{(t)}_G(q,k)|
\leq
\exp\big(2B\left(\gamma/(1-\gamma)\right)^{k/2}\big)\right\}
\,\bigg| \,
\mathcal{E}_{\ref{lemma:small_neighbourhoods}}^{(t)}(q)\Bigg) \geq 1-2\exp(-B).
\label{eq:small_neighbourhoods}
\end{equation}
\begin{proof}
For notational convenience we abbreviate $\gamma=\gamma_{\ref{prop:pa_g}}$ and $A_\gamma=\gamma/(1-\gamma)$.
  We use a path counting technique similar to \cite[Theorem 2]{dereich2012typical},
  see also \cite[Proposition 4.9]{caravenna2016diameter}.
  Consider paths $\pi=(\pi_0, \pi_1, ..., \pi_{k-1}, \pi_K)$
  of length $K<K^\ast_t - M_{\ref{prop:upperbound_dg}}(t)$,
  such that $\pi_0=q, \pi_K=w$, and where $K\leq K^\ast_t - M_{\ref{prop:upperbound_dg}}(t)$.
  For a non-increasing sequence $(\ell_k)_{\geq0}$, we call a path $\pi$ \emph{good} if
  $\pi_k\geq\ell_k$ for all $k\leq K$, and \emph{bad} otherwise.
  Let $\delta':=\delta_{\ref{lemma:small_neighbourhoods}}/3$,
  $\ell_0:=\lceil\delta' t\rceil$.
  We define for a vertex $q\in[t]$ the event
  \begin{equation}\label{eq:event_bad_path}
    \mathcal{E}_{\mathrm{bad}}(q):=\{q<\ell_0\}
    \cup
    \{\exists\text{ bad path of length $K$, for some $K\leq K^\ast_t - M_{\ref{prop:upperbound_dg}}(t)$ from $q$}\}.
  \end{equation}
  The event $\{q< \ell_0\}$ occurs w/p at most $\delta'+1/t$.
  The second event in the union of \eqref{eq:event_bad_path} occurs w/p
  at most $\delta'$, which follows from \cite[first inequality after (18)]{dereich2012typical} for a given choice of $\ell_k$ that we also use here, see \eqref{eq:def_alpha_l} below. Hence,
  for a vertex $q$ chosen uniformly at random from $[t]$
  \[
    \prob(\mathcal{E}_{\mathrm{bad}}(q))\leq2\delta' + 1/t\leq\delta_{\ref{lemma:small_neighbourhoods}},
  \]
  when $t\geq3/\delta_{\ref{lemma:small_neighbourhoods}}$. Let
  $
  \mathcal{E}_{\ref{lemma:small_neighbourhoods}}^{(t)}(q):=\left(\mathcal{E}_{\mathrm{bad}}(q)\right)^c.
  $
  The key element is to prove that for $B$ large enough
  \begin{equation}
    \E[|\partial\mathcal{B}_G(q,k)| \mid \mathcal{E}_{\ref{lemma:small_neighbourhoods}}^{(t)}(q)] \leq \exp\big(BA_\gamma^{k/2}\big).\label{eq:upperbound_ebkv}
  \end{equation}
  Indeed, once we have shown that \eqref{eq:upperbound_ebkv} holds, applying Markov's inequality yields
  \begin{align}
    \prob\Big(
      |\partial\mathcal{B}^{(t)}_G(q,k)| \geq \exp\big(2BA_\gamma^{k/2}\big)
    \mid \mathcal{E}_{\ref{lemma:small_neighbourhoods}}^{(t)}(q)\Big)
    &\,\leq\, \E[|\partial\mathcal{B}_G^{(t)}(q,k)|\mid \mathcal{E}_{\ref{lemma:small_neighbourhoods}}^{(t)}(q)]\,\exp\big(-2BA_\gamma^{k/2}\big) \nonumber\\
    &\,\leq\, \exp\big(-BA_\gamma^{k/2}\big). \label{eq:small_neighbourhoods_applied_markov}
  \end{align}
  Given $\eqref{eq:upperbound_ebkv}$, applying a union bound on
  \eqref{eq:small_neighbourhoods_applied_markov} leads to the result in \eqref{eq:small_neighbourhoods}:
  \begin{align*}
  \prob\bigg(\bigcup_{k=1}^K\big\{|\partial\mathcal{B}^{(t)}_G(q,k)| \geq \exp\big(2BA_\gamma^{k/2}\big)\mid \mathcal{E}_{\ref{lemma:small_neighbourhoods}}^{(t)}(q)\big\}\bigg)
  \leq \sum_{k\in[K]}\exp\big(-BA_\gamma^{k/2}\big)
  \leq 2\exp(-B).
  \end{align*}
  We are left to prove \eqref{eq:upperbound_ebkv}.
  The idea is to  use the number of paths to $\partial \mathcal{B}_G^{(t)}(q,k)$ as an upper bound for the number of
  vertices in $\partial \mathcal{B}_G^{(t)}(q,k)$.
  For $k\leq K^\ast_t-M_{\ref{prop:upperbound_dg}}(t)$, conditionally on $\mathcal{E}_{\ref{lemma:small_neighbourhoods}}^{(t)}$,
  there are only good paths of length $k$ emanating from $q$.
  Let $\Pi_{k, t}^{(\mathrm{g})}(q)$ be the set of such paths. Recall Proposition \ref{prop:pa_g}, and
  observe that we can bound
  \begin{align}
    \E[|\partial \mathcal{B}_G^{(t)}(q,k)| \mid \mathcal{E}_{\ref{lemma:small_neighbourhoods}}^{(t)}(q)]
     \,\leq\, \E\left[\left|\Pi_{k, t}^{(\mathrm{g})}(q)\right| \mid \mathcal{E}_{\ref{lemma:small_neighbourhoods}}^{(t)}(q)\right]
     \,\leq\, \frac{1}{1-\delta_{\ref{lemma:small_neighbourhoods}}}\E\left[\left|\Pi_{k, t}^{(\mathrm{g})(q)}\right|\right].\nonumber   \end{align}
  We can bound $\E[|\Pi_{k, t}^{(\mathrm{g})(q)}|]$ using \eqref{eq:definition_p_v0_v1}, with $p$ defined in \eqref{eq:definition_p_v0_v1}, as
  \[
  \E\left[\left|\Pi_{k, t}^{(\mathrm{g})(q)}\right|\right]
  \leq
  \sum_{w=\lceil\ell_k\rceil}^t\sum_{\pi_1=\ell_1}^t\dots\sum_{\pi_{k-1}=\ell_{k-1}}^t p(q, \pi_1, ..., \pi_{k-1}, w)
  =:
  \sum_{w=\lceil\ell_k\rceil}^t f_{k,t}(q, w)\label{eq:exp_delta_bkv}.
  \]
  Intuitively, when $w\geq \ell_k$, $f_{k,t}(q, w)$ is an upper bound for the expected number of good paths
  from $q$ to $w$ of length $k$.
  From \cite[Section 4.1]{dereich2012typical} it follows that for $q\geq\ell_0$ there is a majorant
  of the form
  \begin{equation}
  f_{k,t}(q, w)\leq \alpha_kw^{-\gamma} + \mathds{1}_{\{w>\ell_{k-1}\}}\beta_kw^{\gamma-1},
  \label{eq:small_neighbourhoods_majorant}
  \end{equation}
  where the sequences $\alpha_k, \beta_k, \ell_k$ are defined recursively as
  \begin{align}
    \alpha_k &:=
    \begin{dcases}
      \nu_{\ref{prop:pa_g}}({\delta'} t)^{\gamma-1} & \text{if } k=1, \\
      c\left(\alpha_{k-1}\log(t/\ell_{k-1}) + \beta_{k-1}t^{2\gamma-1}\right) \phantom{i}& \text{if } k>1,
    \end{dcases}\label{eq:def_alpha_rec} \\
    \beta_k &:=
    \begin{dcases}
      \nu_{\ref{prop:pa_g}}({\delta'} t)^{-\gamma} &\text{if } k=1, \\
      c(\alpha_{k-1}\ell_{k-1}^{1-2\gamma} + \beta_{k-1}\log(t/\ell_{k-1}))\phantom{f}\hspace{2pt} & \text{if } k>1,
    \end{dcases} \label{eq:def_beta_rec}\\
    \ell_k &:=
    \begin{dcases}
      \lceil{\delta'} t\rceil & \text{if } k=0, \\
      \argmax_{x\in\mathbb{N}\backslash\{0, 1\}} \left\{\frac{1}{1-\gamma}\alpha_kx^{1-\gamma}\leq \frac{6{\delta'}}{\pi^2k^2}\right\}&\text{if }k>0,
    \end{dcases}\label{eq:def_alpha_l}
  \end{align}
  with a constant $c=c(\gamma,\nu_{\ref{prop:pa_g}})>1$ chosen in \cite[Lemma 1]{dereich2012typical}.
  Recall $\gamma\in(1/2,1)$.
  Using the above definitions and majorant in \eqref{eq:small_neighbourhoods_majorant}, we return to the bound \eqref{eq:exp_delta_bkv}, and see
  \begin{align}
    \E\left[\left|\Pi_{k, t}^{(\mathrm{g})(q)}\right|\right]
    &\leq \sum_{w=\ell_k}^t\left(\alpha_kw^{-\gamma} + \mathds{1}_{\{w>\ell_{k-1}\}}\beta_kw^{\gamma-1}\right) \nonumber\\
    &=\alpha_k\Big(\ell_k^{-\gamma}+\sum_{w=\ell_k+1}^tw^{-\gamma}\Big) + \beta_k \sum_{w={\ell_{k-1}+1}}^tw^{\gamma-1}. \label{eq:good_paths_alpha_beta}
  \end{align}
  Observe that for $a,b>0$, $\mu\in(0,1)$
  \[
    \sum_{w=a+1}^b w^{-\mu}
    \leq \int_a^b w^{-\mu}\mathrm{d}w
    \leq \frac{b^{1-\mu}}{1-\mu},
  \]
  so that we can bound \eqref{eq:good_paths_alpha_beta} from above by
  \begin{align}
    \E\left[\left|\Pi_{k, t}^{(\mathrm{g})(q)}\right|\right]
    &\leq
    \alpha_k\ell_k^{-\gamma} + \frac{\alpha_k\ell_k^{1-\gamma}}{1-\gamma}(t/\ell_k)^{1-\gamma}
    + \frac{\beta_k}{\gamma}t^\gamma
    =: T_1 + T_2 + T_3.\label{eq:delta_bkv_t123}
  \end{align}
  As a result of \eqref{eq:def_alpha_l}, we obtain for the first term, since $\ell_k\geq1$,
  \begin{equation}
    T_1=\alpha_k\ell_k^{-\gamma}\leq \frac{6\delta'}{\pi^2k^2}(1-\gamma) \ell_k^{-1}\leq \frac{3\delta'}{\pi^2k^2}(1-\gamma)\leq \delta'.\label{eq:delta_bkv_t1}
  \end{equation}
  To bound $T_2$ and $T_3$, we use a claim from \cite[Theorem 2, (19)]{dereich2012typical}
  that the sequence $t/\ell_k$ does not increase \emph{too fast},
  i.e., there exists $B^\ast(\gamma,\nu_{\ref{prop:pa_g}})$
  such that for  $t$ sufficiently large
  \begin{equation}
    t/\ell_k\leq\exp\big(BA_\gamma^{k/2}\big)\label{eq:upperbound_eta}
  \end{equation}
  for any $B\geq B^\ast(\gamma,\nu_{\ref{prop:pa_g}})$. Consider $T_2$, and substitute the bounds from \eqref{eq:def_alpha_l} and \eqref{eq:upperbound_eta}
  \begin{equation}
      \frac{\alpha_k\ell_k^{1-\gamma}}{1-\gamma}(t/\ell_k)^{1-\gamma}
      \leq \frac{6\delta'}{\pi^2k^2}\exp\left(B(1-\gamma)A_\gamma^{k/2}\right)\leq \delta'\exp\left(BA_\gamma^{k/2}\right).\label{eq:delta_bkv_t2}
  \end{equation}
  For $T_3$, we substitute \eqref{eq:def_beta_rec} for $\beta_k$ to obtain
  \begin{align}
    T_3
    =\frac{1}{\gamma}\beta_kt^\gamma
    = \frac{c}{\gamma}\alpha_{k-1}(t/\ell_{k-1})^\gamma\ell_{k-1}^{1-\gamma} + \frac{c}{\gamma}\beta_{k-1}t^\gamma\log(t/\ell_{k-1})
    =: T_{31} + T_{32}. \label{eq:delta_bkv_t3}
  \end{align}
  We use \eqref{eq:def_alpha_l} and \eqref{eq:upperbound_eta} to bound
  \begin{align}
    T_{31} = \frac{c}{\gamma}(t/\ell_{k-1})^\gamma\alpha_{k-1}\ell_{k-1}^{1-\gamma}
    \leq \frac{6c\delta'(1-\gamma)}{\gamma\pi^2(k-1)^2}\exp\left(B\gamma A_\gamma^{(k-1)/2}\right)
    \leq \frac{c\delta'}{\gamma}\exp\left(BA_\gamma^{k/2}\right). \label{eq:delta_bkv_t31}
  \end{align}
  Observe that by \eqref{eq:def_alpha_rec}, $\beta_{k-1}\leq t^{1-2\gamma}\alpha_k/c$.
  Hence, the second term in \eqref{eq:delta_bkv_t3} is bounded by
  \[
   T_{32} =\frac{c}{\gamma}\beta_{k-1}t^\gamma\log(t/\ell_{k-1})\leq \frac{1}{\gamma}t^{1-\gamma}\alpha_k\log(t/\ell_{k-1}).
  \]
  By \eqref{eq:def_alpha_l}, if $\delta'<\pi^2/6$, then $\alpha_k\leq (1/\ell_k)^{1-\gamma}$.
  After substituting  \eqref{eq:upperbound_eta} for $B$ sufficiently large
  \begin{align}
    T_{32} \leq \frac{1}{\gamma}(t/\ell_k)^{1-\gamma}\log(t/\ell_{k-1})
    \leq \frac{B}{\gamma}A_\gamma^{(k-1)/2}\exp\left(B(1-\gamma)A_\gamma^{k/2}\right).\nonumber
  \end{align}
  As the first factor grows exponentially in $k$, by increasing $B$, we obtain
  $
    T_{32}
    \leq
    \exp\big(BA_\gamma^{k/2}\big).
  $
  Combining this with \eqref{eq:delta_bkv_t123}, \eqref{eq:delta_bkv_t1}, \eqref{eq:delta_bkv_t2},  \eqref{eq:delta_bkv_t3}, and \eqref{eq:delta_bkv_t31} gives us that there exists
  a constant $B_{\ref{lemma:small_neighbourhoods}} >
  B^\ast(\gamma, \nu_{\ref{prop:pa_g}})>0$
  such that for $B\geq B_{\ref{lemma:small_neighbourhoods}}$ we obtain the desired bound \eqref{eq:upperbound_ebkv}.
\end{proof}
\end{lemma}
We are ready to prove the lower bound of the conservative case.
\subsubsection*{Proof of Proposition \ref{prop:lowerbound_unexplosive}}
  Recall $\mathcal{E}_{\ref{prop:upperbound_dg}}^{(t)}$ and $\mathcal{E}_{\ref{lemma:small_neighbourhoods}}^{(t)}(q)$
  and their error probabilities $\delta_{\ref{prop:upperbound_dg}}, \delta_{\ref{lemma:small_neighbourhoods}}$ in \eqref{eq:upperbound_dg}, \eqref{eq:small_neighbourhoods_condition}. First we introduce some notation to work with FPA, VPA, and GVPA simultaneously. Eventually we dinstinguish
  the model GVPA vs.\ the models FPA and GVPA again.
  Recall $\varepsilon$ from the statement of Proposition \ref{prop:lowerbound_unexplosive}.
  Observe that if $t$ is sufficiently large, by Proposition \ref{prop:upperbound_dg},
  \[
  \prob\big(d_G^{(t)}(u,v) \geq 2(1-\varepsilon/2)K^\ast_t\big) \geq 1-\delta_{\ref{prop:upperbound_dg}}.
  \]
  In order to apply Lemma \ref{lemma:small_neighbourhoods}, relying on Proposition \ref{prop:pa_g}, for GVPA we set $\gamma_{\ref{prop:pa_g}}$ as the solution of
  \begin{equation}
  \frac{\log\left(\gamma_{\ref{prop:pa_g}}/(1-\gamma_{\ref{prop:pa_g}})\right)}{|\log\left(\tau-2\right)|}=\frac{1}{1-\varepsilon/2},
  \label{eq:prop_proof_lower_gamma}
  \end{equation}
  so that indeed $\gamma_{\ref{prop:pa_g}}>1/(\tau-1)$ as required for Proposition \ref{prop:pa_g}.
  For FPA and VPA, we set $\gamma_{\ref{prop:pa_g}}=1/(\tau-1)$, where $\tau$ is the power-law exponent of the considered model,
  as defined in Definition \ref{def:pa_m_delta} and Definition \ref{def:pa_gamma}.
  To avoid double notation, we define
  \begin{equation}
    K'_t :=
    \begin{dcases}
      \lfloor(1-\varepsilon/2)K_t^\ast\rfloor,&\text{for $\mathrm{GVPA}(f)$,}\\
      K^\ast_t - M_{\ref{prop:upperbound_dg}}(t),&\text{for FPA and VPA.}
    \end{dcases}
    \label{eq:lowerbound_unexplosive_kprime_t}
  \end{equation}
  Apply Proposition \ref{prop:upperbound_dg} and Lemma \ref{lemma:small_neighbourhoods} for $\delta_{\ref{prop:upperbound_dg}}=\delta_{\ref{lemma:small_neighbourhoods}}=\delta/6$
  so that w/p at least $1-\delta/6$ the neighbourhoods are disjoint and not \emph{too large} for $k<K'_t$.
  Assume that $t$ is so large that
  \begin{equation}
  \prob\big(\mathcal{E}_{\ref{prop:upperbound_dg}}^{(t)}\cap \mathcal{E}_{\ref{lemma:small_neighbourhoods}}^{(t)}(u)\cap \mathcal{E}_{\ref{lemma:small_neighbourhoods}}^{(t)}(v)\big)\geq 1-\delta/2,\label{eq:error_prob_condition_13}
  \end{equation}
  by a union bound on the complements of the events between brackets.
  Throughout the remainder of this proof, we write $A=\gamma_{\ref{prop:pa_g}}/(1-\gamma_{\ref{prop:pa_g}})$
  and condition on the above event between brackets.
  On this event, the bound in \eqref{eq:weighted_distance_lb_sum_min} holds.  We now bound the minimal weights in \eqref{eq:weighted_distance_lb_sum_min}.
  Assume $B>B_{\ref{lemma:small_neighbourhoods}}$.
  By \eqref{eq:small_neighbourhoods} in Lemma \ref{lemma:small_neighbourhoods}, for $q\in\{u,v\}$, between $\partial \mathcal{B}^{(t)}_{G}(q,k)$ and
 $\partial \mathcal{B}^{(t)}_{G}(q,k+1)$ there are at most
 \begin{align}
 |\partial \mathcal{B}^{(t)}_{G}(q,k)|\cdot|\partial \mathcal{B}^{(t)}_{G}(q,k+1)|
 &\,\leq\,
 \exp\big(2B\big(1 + \sqrt{A}\big)A^{k/2}\big)
 \,=:\,n_k \nonumber%\label{eq:nk-lower}
\end{align}
 edges with cumulative error probability (over $k$) at most $2\exp(-B)=:\delta_{B}^{(1)}.$
 By ($\ast$) in \eqref{eq:minimum-iid-lower},
 \begin{align}\label{eq:upperbound_small_min}
   \prob\Bigg(
   \bigcup_{k\in[K'_t}
   \Big\{\min_{j\in[n_k]}L_{j,k}
   \leq F_L^{(-1)}\big(n_k^{-1-\xi}\big)\Big\}\Bigg)
   \leq \sum_{k\in [K'_t]} \exp\big(-2B\xi\big(1 + \sqrt{A}\big) A^{k/2}\big) \nonumber\\
   \leq 2\exp\big(-2B\xi\big(1 + \sqrt{A}\big)\big)=:\delta_{B}^{(2)},
 \end{align}
following from a union bound over $k\in[K'_t]$. Combining \eqref{eq:weighted_distance_lb_sum_min} and \eqref{eq:upperbound_small_min}
 gives for $q\in\{u,v\}$
 \begin{equation}\label{eq:lowerbound_dl_sum}
   \prob\bigg(
   d_L^{(t)}(q, \partial \mathcal{B}_{G}^{(t)}(q, K'_t))
   \leq
   \sum_{k\in[K'_t]} F_L^{(-1)}\left(\exp\big(-2B(1+\xi)\big(1 + \sqrt{A}\big)A^{k/2}\big)\right)\bigg)\leq \delta_B^{(1)}+\delta_B^{(2)},
 \end{equation}
 when $B>B_{\ref{lemma:small_neighbourhoods}}$.
 In particular, we choose $B$ so that
 $\delta_B^{(1)}$ + $\delta_B^{(2)} \leq \delta/4$.
 Recall \eqref{eq:error_prob_condition_13} and the reasoning before \eqref{eq:weighted_distance_lb_sum_min}, so that by
 the law of conditional probability and a union bound for $q\in\{u,v\}$
 on the event in \eqref{eq:lowerbound_dl_sum}, we obtain w/p at least $1-\delta$ that
 \begin{equation}
 d^{(t)}_L\left(u, v\right) \geq
 2\sum_{k\in[K'_t]} F_L^{(-1)}\left(\exp\big(-2B(1+\xi)\big(1 + \sqrt{A}\big)A^{k/2}\big)\right) =: 2S_{K'_t}.
 \label{eq:r_sb}
 \end{equation}
 It remains to show that $2S_{K'_t}$ is larger than the rhs between brackets in \eqref{eq:weighted_distance_nonexplosive}
 and \eqref{eq:weighted_distance_nonexplosive_tight} for the corresponding models.
 Similarly to the upper bound, we rewrite the sum in $S_{K'_t}$ to match the summands in
 $Q_t$ in \eqref{eq:k_ast_qt}, then we bound the sum from below by switching to integrals, apply a variable transformation, and go back to sums.
 Applying ($\ast$) in \eqref{eq:int_sum_bounds} to $S_{K'_t}$ yields
 \begin{equation}
 S_{K'_t}
 \geq \int_{1}^{K'_t}F_L^{(-1)}\left(\exp\big(-2B(1+\xi)\big(1 + \sqrt{A}\big)A^{x/2}\big)\right)\mathrm{d}x.
 \label{eq:prop_lower_before_variable_trans}
 \end{equation}
 For the models FPA and VPA, $A=1/(\tau-2)$, while for GVPA, $A=(1+\varepsilon')/(\tau-2)$, for some $\varepsilon'=\varepsilon'(\varepsilon)$.
 After the variable transformation $2B(1+\xi)(1+\sqrt{A})A^{x/2}=(\tau-2)^{-y/2}$ on the rhs of \eqref{eq:prop_lower_before_variable_trans},
 the function over which we integrate matches the function in the sum in $Q_t$ in \eqref{eq:k_ast_qt}, i.e.,
 \begin{equation}
    S_{K'_t} \geq \frac{1}{r}\int_{r+s(B)}^{rK'_t + s(B)}F_L^{(-1)}\left(\exp\big(-(\tau-2)^{-y/2}\big)\right)\mathrm{d}y,
    \label{eq:prop_lower_after_variable_trans}
 \end{equation}
 where $
   s(B):= 2\log\big(2B(1+\xi)(1+\sqrt{A})\big)/|\log\left(\tau-2\right)|
 $, and by our choice of $\gamma$ in \eqref{eq:prop_proof_lower_gamma}
\begin{align}
  r&:=
  \begin{dcases}
    \frac{\log(A)}{|\log\left(\tau-2\right)|}=\frac{1}{1-\varepsilon/2},& \text{for the model GVPA,} \\
  1,&\text{for the models $\mathrm{FPA}, \mathrm{VPA}$.}\label{eq:lower_r_variable_trans}
\end{dcases}
\end{align}
Apply $(\star)$ from \eqref{eq:int_sum_bounds} to the rhs of \eqref{eq:prop_lower_after_variable_trans} gives, abbreviating $a_k:=F_L^{(-1)}\left(\exp(-(\tau-2)^{-k/2}\right)$,
\begin{equation}
  S_{K'_t}
  \geq
  \frac{1}{r}\sum_{k=\lceil r + s(B)\rceil + 1}^{\lfloor rK'_t + s(B)\rfloor}a_k.
  \nonumber%\label{eq:lowerbound_unexplosive_ak}
\end{equation}
For FPA and VPA,
using $r=1$ and $K'_t$ in \eqref{eq:lowerbound_unexplosive_kprime_t}, we bound
\begin{equation}
  S_{K'_t}\geq \sum_{k=\lceil s(B)\rceil + 2}^{K^\ast_t - M_{\ref{prop:upperbound_dg}}(t) + \lfloor s(B)\rfloor}a_k=:\widetilde{Q}_t.
  \nonumber%\label{eq:lowerbound_unexplosive_affine}
\end{equation}
Up to a shift in the boundaries, the above summands match the summands in $Q_t$ in \eqref{eq:k_ast_qt}.
To obtain \eqref{eq:weighted_distance_nonexplosive_tight}, we should choose  $M_{\ref{prop:lowerbound_unexplosive}}$ such that
$
Q_t - M_{\ref{prop:lowerbound_unexplosive}} \leq \widetilde{Q}_t,
$
which is equivalent to
\[
M_{\ref{prop:lowerbound_unexplosive}} \geq Q_t - \widetilde{Q}_t
=
\sum_{k=1}^{K^\ast_t}a_k - \sum_{k=\lceil s(B)\rceil + 2}^{K^\ast_t - M_{\ref{prop:upperbound_dg}}(t) + \lfloor s(B)\rfloor}a_k = \sum_{k=1}^{\lceil s(B)\rceil + 1}a_k + \sum_{k=K^\ast_t + 1}^{K^\ast_t-M_{\ref{prop:upperbound_dg}}(t) + \lfloor s(B)\rfloor}a_k,
\]
where we define the second sum as 0 if $\lfloor s(B)\rfloor - M_{\ref{prop:upperbound_dg}}<1$.
The first sum on the rhs can be bounded by a constant. As the sequence $a_k$ is decreasing, we shift the summation boundaries of the second sum on the rhs and choose
\[
M_{\ref{prop:lowerbound_unexplosive}}:=\sum_{k=1}^{\lceil s(B)\rceil + 1}a_k + \sum_{k= 1}^{-M_{\ref{prop:upperbound_dg}}(t) + \lfloor s(B)\rfloor}a_k.
\]
Observe that we can bound the second sum here by a constant, since $M_{\ref{prop:upperbound_dg}}(t)= O(1)$ for FPA and VPA by Proposition \ref{prop:upperbound_dg}.
As a result, we bound the lhs \eqref{eq:r_sb} from above and obtain the result
\begin{align*}
  \prob\big(d_L^{(t)}(u, v) \geq 2(Q_t - M_{\ref{prop:lowerbound_unexplosive}}) \big)
  \geq
\prob\big(d_L^{(t)}(u, v) \geq 2\widetilde{Q}_t \big)
\geq
\prob\big(d_L^{(t)}(u, v)\geq 2 S_{K'_t}\big)
\geq
1-\delta.
\end{align*}

We turn now to GVPA.
Recalling $K'_t$ from \eqref{eq:lowerbound_unexplosive_kprime_t} and $r$ from \eqref{eq:lower_r_variable_trans}, we bound $rK'_t$ in the upper summation boundary in \eqref{eq:r_sb} from below by
\begin{align*}
\frac{1}{1-\varepsilon/2}\left\lfloor\left(1-\frac{\varepsilon}{2}\right)K^\ast_t\right\rfloor
&\geq
\frac{1}{1-\varepsilon/2}\left(\left(1-\frac{\varepsilon}{2}\right)K^\ast_t-1\right)
\geq
K^\ast_t - \frac{1}{1-\varepsilon/2}.
\end{align*}
Thus, we further bound $S_{K'_t}$ by
\begin{equation}
S_{K'_t}\geq (1-\varepsilon/2)\sum_{k=\lceil 1/(1-\varepsilon/2) + s(B)\rceil + 1}^{\left\lfloor K^\ast_t - 1/(1-\varepsilon/2) + s(B)\right\rfloor}a_k.\nonumber\end{equation}
We rewrite the boundaries of the sum and use that $\varepsilon<1$ and hence $1/(1-\varepsilon/2)\leq 2$, so
\begin{align}
  S_{K'_t}
  &\geq
  (1-\varepsilon/2)\Bigg(\sum_{k\in[K^\ast_t]} a_k + \sum_{k=K^\ast_t + 1}^{\left\lfloor K^\ast_t - 1/(1-\varepsilon/2) + s(B)\right\rfloor}a_k - \sum_{k\in[\lceil 1/(1-\varepsilon/2) + s(B)\rceil]}a_k\Bigg) \nonumber\\
  &\geq
  (1-\varepsilon/2)\Bigg(\sum_{k\in[K^\ast_t]} a_k - \sum_{k\in[2+\lceil s(B)\rceil]}a_k\Bigg),
  \label{eq:prop_lower_changed_summation_bounds}
\end{align}
As the second sum in \eqref{eq:prop_lower_changed_summation_bounds} is  a constant, it can be bounded by $\varepsilon/2$ times the first sum in \eqref{eq:prop_lower_changed_summation_bounds} when $t$ is sufficiently large.
Thus,  we have that
\[
d^{(t)}_L\left(u, v\right)
\geq
2(1-\varepsilon/2)^2\sum_{k\in[K^\ast_t]}a_k
\geq
2(1-\varepsilon)\sum_{k\in[K^\ast_t]}a_k
=
2(1-\varepsilon)Q_t
\]
w/p $1-\delta$, for any fixed $B$ and $t$ sufficiently large. This is the asserted bound in \eqref{eq:weighted_distance_nonexplosive}.
\qed

To finish the section, we prove the lower bound for the explosive class.
\begin{proof}[Proof of Proposition \ref{prop:lowerbound_explosive}]
  Recall $K^\ast_t$ from \eqref{eq:k_ast_qt} and let $M_{\ref{prop:upperbound_dg}}(t)$ be the function we obtain by applying Proposition \ref{prop:upperbound_dg}
  for $\delta_{\ref{prop:upperbound_dg}} = \delta / 2$. As a result, we obtain for any $K\leq K^\ast_t-M_{\ref{prop:upperbound_dg}}(t)$,
  \[
  \prob\bigg(d_L^{(t)}(u, v) < \sum_{q\in\{u,v\}} d_L^{(t)}(q, \mathcal{B}_G^{(t)}(q, K)) \bigg)
  \leq \frac{\delta}{2}.
  \]
  Let $a(t) := \min\{\kappa_{\delta/6}(t), K^\ast_t-M_{\ref{prop:upperbound_dg}}(t)\}$,
  where $\kappa_{\delta/4}(t)$ denotes the maximum number of generations in LWL to maintain a coupling with $\mathrm{PA}_t$
  w/p at least $1-\delta/4$, as in Proposition \ref{prop:local_weak_limit} and \ref{prop:local_weak_limit}. Hence, we obtain
  $d_L^{(t)}\big(q, \partial \mathcal{B}_G(q,a(t))\big)= \beta^{\mathrm{LWL}^{(q)}_{a(t)}}_{a(t)}$ if the coupling is successful for $q\in\{u, v\}$
  and thus
  \[
  \prob\bigg(d_L^{(t)}(u, v) < \sum_{i\in\{1,2\}} \beta_{a(t)}^{\mathrm{LWL}_{a(t)}^{(i)}} \bigg)
  \leq \delta,
  \]
  which finishes the proof.
\end{proof}

\section{Hopcount}\label{sec:hopcount}
In this section we prove Theorem \ref{thm:hopcount_nonexplosive}. It follows from an
adaptation of the proof of Theorem \ref{thm:weighted_distance_nonexplosive}.
\begin{proof}[Proof of Theorem \ref{thm:hopcount_nonexplosive}]
Since the hopcount is at least the graph distance, Proposition \ref{prop:upperbound_dg} implies that for any $\varepsilon>0$
\begin{equation}
\prob\big(d_H^{(t)}(u, v)\geq2(1-\varepsilon)K^\ast_t\big)\longrightarrow 1,\qquad\text{as }t\rightarrow\infty,
\label{eq:prop_proof_hopcount_lowerbound}
\end{equation}
for the model $\mathrm{GVPA}(f)$, and shows lower tightness for FPA and VPA.
It suffices to prove the matching upper bounds.
By rescaling the weights to $L/a$, all weights are at least one.
For any two vertices $u$ and $v$, the shortest path with the unscaled weights uses the same edges as the shortest path with the scaled weights.
Hence, $d_H^{(t)}(u,v)\leq d_{L/a}^{(t)}(u,v)$. Observe that for any $x>0$,
$
F_{L/a}^{(-1)}(x) = 1 + F_{L/a-1}^{(-1)}(x).
$
Fix a small $\delta>0$ and recall steps (i), (iii), and (iv) from the proof of Proposition \ref{prop:upper_bound}.
Analogously to the reasoning leading to \eqref{eq:prop_proof_qt_result},
for any $s_0$ there exists constants $C_{\ref{prop:graph_distance_to_high_degree}}, C_{\ref{prop:innercore_bounded}}, M$
such that for $t$ sufficiently large
\begin{equation}
  \prob\bigg(d_{L/a}^{(t)}(u, v)
  \leq
  2K^\ast_t  + 2\sum_{k=h_{\tau}(s_0)}^{h_{\tau}(s_0) + K^\ast_t}F_{L/a-1}^{(-1)}\left(\exp\left(\tau-2\right)^{-k/2}\right) + M
  \bigg)
\geq 1-\delta/2,
\nonumber%\label{eq:upperbound_hopcount_starting_bound}
\end{equation}
where we used the upper bound in \eqref{eq:weighted_distance_in_t} from Proposition \ref{prop:weighted_distance_in_t} on the weight of the segment reaching the inner core. Thus, we did \emph{not} shift the summation bounds as in step (ii) in the proof of Proposition \ref{prop:upper_bound}.
If $I(L/a-1)<\infty$, then the above sum is finite and thus yields the result for tightness in \eqref{eq:thm_hopcount_tight}.
If the sum is infinite, we choose $s_0$ so large that all the terms
are bounded by a fixed $\varepsilon/2>0$, hence the sum is bounded by $\varepsilon K^\ast_t/2$.
As $K^\ast_t$ is increasing in $t$, for $t$ sufficiently large,
$\varepsilon K^\ast_t > M$. Thus,
\[
\prob\left(d_H^{(t)}(u,v)\leq 2(1+\varepsilon)K^\ast_t\right)\geq \prob\left(d_{L/a}^{(t)}(u, v) \leq 2(1+\varepsilon)K^\ast_t\right)\geq 1-\delta.
\]
Combining this upper bound with the lower bound \eqref{eq:prop_proof_hopcount_lowerbound} yields
the desired asymptotics in \eqref{eq:thm_hopcount_nontight}.
\end{proof}

\section{Conservative weights on the LWL}\label{sec:cons-lwl}
\subsection*{Proof of Theorem {\ref{theorem:conservative_lwl}}}
We prove \eqref{eq:lwl-tightness} in Theorem \ref{theorem:conservative_lwl} for the LWL of FPA and VPA such that the weight distributions satisfies \eqref{eq:tightness_cond}. At the end of the proof we describe briefly how to prove \eqref{eq:lwl-no-tightness}.
The proof is split in an upper bound and a lower bound. Both bounds rely on a coupling of the LWL to the finite graph for some large $t$ and follow from adaptations of the results in Section \ref{sec:upperbound} and Section \ref{sec:lowerbound}. Fix $k$ and $\delta>0$.

For the upper bound we show that there exists $M=M(\delta)$, not depending on $k$, such that
\begin{equation}
\prob\big(\beta_k^{(\text{LWL})}(\circledcirc)\leq Q^\sss{(k)} + M \big)\geq 1-\delta/2, \quad \text{ where }
Q^\sss{(k)}:= \sum_{i=1}^k F_L^{(-1)}\left(\exp\big(-(\tau-2)^{-i/2}\big)\right).
\label{eq:upperbound-lwl-cons}
\end{equation}
Recall $Q_t$ in \eqref{eq:k_ast_qt} and observe that $Q_t=Q^\sss{(K^\ast_t)}$.
Set $\delta_{\ref{prop:local_weak_limit}}=\delta_{\ref{prop:graph_distance_to_high_degree}}=\delta/12$,
$\delta_{\ref{prop:weighted_distance_in_t}}=\delta/6$, and
$\alpha=1-\delta/12$.
Let $t$ be sufficiently large
such that all the following hold.
\begin{enumerate}[label=(\roman*),leftmargin=2\parindent]
  \item By Proposition \ref{prop:local_weak_limit} the LWL and the graph neighbourhood of a typical vertex $q$ can be coupled up to generation $k$ with probability at least $1-\delta/12$. Recall $\beta_k^{(\text{LWL})}(\circledcirc)$ from Definition \ref{def:explosive_graph}. Conditionally on the event that this coupling is successful, $\beta_k^{(\text{LWL})}(\circledcirc)=d_L^{(t)}(q, \partial \mathcal{B}^{(t)}_G(q, k))$. Thus,  the upper bound in \eqref{eq:upperbound-lwl-cons} follows if we show that
 \begin{equation*}
 \prob\big(d_L^{(t)}(q, \partial \mathcal{B}^{(t)}_G(q, k)\big)\leq Q^\sss{(k)} + M \mid \{\text{coupling successful}\}\big)\geq 1-11\delta/12.
 \end{equation*}
 Since $q$ is a typical vertex, the event $\{q<\alpha t\}$ holds with probability at least $1-11\delta/12$. Hence, by a union bound, the event $\mathcal{E}^{\text{(i)}}:=\{q<\alpha t\}\cap\{\text{coupling successful}\}$ holds with probability at least $1-5\delta/6$, leaving to show that
 \begin{equation}
 \prob\big(d_L^{(t)}(q, \partial \mathcal{B}^{(t)}_G(q, k)\big)\leq Q^\sss{(k)} + M \mid \mathcal{E}^{\text{(i)}}\big)\geq 1- \delta/3.
 \label{eq:upperbound-lwl-cons-graph}
 \end{equation}
\item By Proposition \ref{prop:graph_distance_to_high_degree}
for $s_{\ref{prop:graph_distance_to_high_degree}}=s_0$ that we choose below in (iv),
there exists $C_{\ref{prop:graph_distance_to_high_degree}}=C_{\ref{prop:graph_distance_to_high_degree}}(\delta, s_0)$,
such that, for $\mathcal{E}^{(\text{ii})}_1:=\{\exists q'\in \mathcal{B}_G^{((1-\alpha)t)}(q, C_{\ref{prop:graph_distance_to_high_degree}}): D_{(1-\alpha)t}(q')\geq s_0\}$,
$\prob(\mathcal{E}^{(\text{ii})}_1 \mid \mathcal{E}^{(\text{i})})\geq 1-\delta/12.$
All edges on a possible path of length $C_{\ref{prop:graph_distance_to_high_degree}}$ from $q$ to $q'$ are equipped with i.i.d.\ copies of $L$.
Hence, there exists $M^\text{(ii)}=M^\text{(ii)}(\delta,C_{\ref{prop:graph_distance_to_high_degree}})>0$ such that
\[
\prob\bigg(\sum_{i=1}^{C_{\ref{prop:graph_distance_to_high_degree}}}L_i\leq M^\text{(ii)}\bigg)\geq 1-\delta/12.
\]
Let $\mathcal{E}^{(\text{ii})}:=\{\exists q'\in \mathcal{B}_G^{((1-\alpha)t)}(q, C_{\ref{prop:graph_distance_to_high_degree}}): d_L^{((1-\alpha)t)}(q,q')\leq M^\text{(ii)}\}\cap \mathcal{E}^{\text{(ii)}}_1$.
By a union bound we obtain that
\[\prob(\mathcal{E}^{(\text{ii})}\mid\mathcal{E}^{\text{(i)}})\geq 1- \delta/6.\]
\item With $K_t$ in \eqref{eq:kt}, the inequality $C_{\ref{prop:graph_distance_to_high_degree}} + 2K_t \geq k$ holds for all sufficiently large $t$. Then, we construct a greedy path to  the inner core as described in the proof of Proposition \ref{prop:weighted_distance_in_t}.
\item By Proposition \ref{prop:weighted_distance_in_t}
there is an $s_0=s_0(\delta, \varepsilon_{\ref{prop:weighted_distance_in_t}})>0$, such that for  $q'\in[\alpha t]$ with $D_{\alpha t}(q')\geq s_0$, the weighted distance to
the inner core in $\mathrm{PA}_t$ is not too large. The graph $\mathrm{PA}_t$ is coupled to LWL for at least $k$ generations.  Hence, if the coupling succeeds, then any path of $k-C_{\ref{prop:graph_distance_to_high_degree}}$ many edges
emanating from $q'$ is also present in the LWL, as $q'$ is at graph distance at most $C_{\ref{prop:graph_distance_to_high_degree}}$ from $q$ by (ii). Hence, we can use the first $k-C_{\ref{prop:graph_distance_to_high_degree}}$ edges of the greedy path described in Definition \ref{def:layers_connectors_greedy} and follow the proof of Proposition \ref{prop:weighted_distance_in_t} to estimate its total weight. By (iii), $k\leq 2K_t + C_{\ref{prop:graph_distance_to_high_degree}}$, so
ultimately, we obtain that for some that $M^\text{(iv)}>0$, that is not dependent on $k$, that
\begin{equation}
\prob\Big(d_L^{(t)}\big(q',\partial \mathcal{B}^{(t)}_G(q', 0\vee(k-C_{\ref{prop:graph_distance_to_high_degree}}))\big) \leq Q^\sss{(k)} +  M^\text{(iv)}\Big) \geq 1- \delta/6.
\nonumber%\label{eq:prop_proof_distance_to_inner_t_2}
\end{equation}
Denote the above event between brackets by $\mathcal{E}^{(\text{iv})}$.
\end{enumerate}
On the event $\mathcal{E}^{(\text{i})}\cap \mathcal{E}^{(\text{ii})}\cap\mathcal{E}^{(\text{iv})}$, there is a path from $q$ to a vertex in $\mathcal{B}^{(t)}_G(q', k)$, whose total weight is bounded from above by $Q^\sss{(k)} +  M^\text{(ii)}+  M^\text{(iv)}$. If the coupling succeeds, this path is also present in the LWL by Proposition \ref{prop:local_weak_limit}.

To conclude the proof of \eqref{eq:upperbound-lwl-cons}, we recall that showing \eqref{eq:upperbound-lwl-cons} was reduced to showing \eqref{eq:upperbound-lwl-cons-graph}, which
holds for $M=M^\text{(ii)}+  M^\text{(iv)}$, since $\prob( \mathcal{E}^{(\text{ii})}\cap\mathcal{E}^{(\text{iv})}\mid \mathcal{E}^{(\text{i})})\geq 1-\delta/3$ by a union bound.

We proceed to the lower bound of Theorem \ref{theorem:conservative_lwl}, i.e., we show that
for some $M=M(\delta)$
\begin{equation}
\prob\big(\beta_k^{(\text{LWL})}(\circledcirc)\leq Q^\sss{(k)} - M \big)\geq 1- \delta/2.
\label{eq:lowerbound-lwl-cons}
\end{equation}
Let $\delta_{\ref{lemma:small_neighbourhoods}}=\delta/8$, and let $B$ in Lemma \ref{lemma:small_neighbourhoods} be sufficiently large.
Let $t$ be so large that
\begin{enumerate}
  \item $K^\ast_t-M_{\ref{lemma:small_neighbourhoods}}(t)\geq k$.
  \item by Proposition \ref{prop:local_weak_limit} the local weak limit can be coupled to the neighbourhood of radius $k$ of a typical vertex $q$ with probability at least $1-\delta/8$. Conditionally on the event that this coupling is successful, $\beta_k^{(\text{LWL})}(\circledcirc)=d_L^{(t)}(q, \partial \mathcal{B}^{(t)}_G(q, k))$.
\end{enumerate}
 Hence, by (2) it is for \eqref{eq:lowerbound-lwl-cons} sufficient to show that there exists $M$ such that
 \begin{equation*}
   \prob\big(d_L^{(t)}\big(q,\partial B^{(t)}_G(q,k)\big)\leq Q^\sss{(k)} - M \mid \{\text{coupling successful}\}\big)\geq 1- 3\delta/8.
 \end{equation*}
  This follows from easy modifications of the proof of Proposition \ref{prop:lowerbound_unexplosive}. We leave it to the reader to fill in the details.

  Now that \eqref{eq:upperbound-lwl-cons} and \eqref{eq:lowerbound-lwl-cons} have been established, \eqref{eq:lwl-tightness} follows for weight distributions satisfying \eqref{eq:tightness_cond}.
  For other weight distributions, \eqref{eq:lwl-no-tightness} can be proven for FPA and VPA using the same steps.
  For the model GVPA, one can prove \eqref{eq:lwl-no-tightness} using the same couplings from GVPA to VPA as in the proofs of Proposition \ref{prop:upper_bound} and Proposition \ref{prop:lowerbound_unexplosive}. We leave it to the reader to fill in the details. \qed

%%%%%%%%%%%%%%%%%%%%%%%%%%%%%%%%%%%%%%%%%%%%%%%%%%%%%%%%%%%%%%%%%%%%%%%%%%%%%%%

\appendix
\section{Weighted distance in the inner core}
\subsection*{Proof of Proposition \ref{prop:innercore_bounded}}\label{proof:innercore_bounded}
  We give a coupling proof, similar to \cite[Proposition 3.2]{dommers2010diameters}.
  We construct a path from $w_1$ to $w_2$ via a subset of the inner core,
  in which we can bound the weighted distance between two vertices whp.
  We show the latter first, after which we show that if $w_1$ or $w_2$ is not contained in this subset,
  the (weighted) distance to this subset is also small.

  By Lemma \ref{lemma:total_degree_high_degree_vertices}, there are at least $n_t=\lfloor\sqrt{t}\rfloor$ vertices in the inner core.
  Let $I$ be the set of the first $n_t$ vertices that have
  degrees at least $(\alpha t)^{1/(2(\tau-1))}\log(\alpha t)^{-1/2}$.
  We construct a graph $H_t$ on these vertices as follows. Recall the definition of an $\alpha$-connector from Definition \ref{def:layers_connectors_greedy}. Let $i,j$ be connected in $H_t$
  if there exists an $\alpha$-connector $y$. The weight on the edge $(i,j)$ is $L_{(i, y)} + L_{(j, y)}$.
  As explained in the proof of \cite[Proposition 3.2]{dommers2010diameters} for the model FPA, $H_t$ stochastically dominates a
  dense uniform  Erd\H{o}s-R\'enyi graph $G(n_t,p_t)$, where
  \[
  p_t := \frac{t^{\frac{1}{\tau-1}-1}}{2\log^2t}.
  \]
  Here, we say that a random graph $G$ dominates a random graph $H$ if there exists a coupling such that every edge in $H$ is also contained in $G$.
  Using \cite[Theorem 1.1]{dereich2009random}, one can verify that the same holds for GVPA.
  In \cite[Chapter 10.2]{bollobas2001random} the diameter of the dense ERRG is discussed, and it is shown that the diameter is bounded.
  Hence, the first assertion \eqref{eq:bounded_distance_inner_graph} follows. From now on, we assume that $F_L(x)>0$ for all $x>0$.
  Let $\Delta=\Delta(\tau)$ denote the diameter of the $G(n_t,p_t)$.
  The proof techniques in the above mentioned book chapter rely on the exploration around two vertices.
  In particular, it can be derived that the number of disjoint paths  between two vertices of length $\Delta$ tends to infinity with the size of the graph.
  Hence, there is a function $r_t$ tending to infinity with $t$, such that there are at least  $r_t$ disjoint paths between $u$ and $v$.
  The weight on the $i$-th path is distributed as $L_{\star}^{(i)} := L_1^{(i)} + ... + L_{2\Delta}^{(i)}$. As $F_L(x)>0$ for any $x>0$, the same holds for $F_{L_{\star}}$.
  Thus,
  \[
  \lim_{t\rightarrow\infty}\prob\big(\min_{i\in r_t} L_{\star}^{(i)} > \varepsilon \big)=0,
  \]
  for any $\varepsilon>0$.
  Hence, for $w'_1,w'_2 \in I$ and any $\varepsilon_{\ref{prop:innercore_bounded}},\delta_{\ref{prop:innercore_bounded}}>0$,
  \begin{equation}\label{eq:innercore-bound-1}
  \prob\left(d_L^{(t)}(w'_1,w'_2)\geq \varepsilon_{\ref{prop:innercore_bounded}}/3\right)\leq \delta_{\ref{prop:innercore_bounded}}/3.
  \end{equation}
  Assume that there is an $i\in\{1, 2\}$ such that $w_i\notin I$, and observe that we are done if we prove
  \begin{equation}\label{eq:innercore-bound-2}
  \prob\left(d_L^{(t)}(w_i, I)\geq \varepsilon_{\ref{prop:innercore_bounded}}/3\right)\leq \delta_{\ref{prop:innercore_bounded}}/3.
  \end{equation}
  Analogously to the proof of Lemma \ref{lemma:error_prob_tk},
  one can verify that whp the number of $\alpha$-connectors between $w_i$ and $I$ tends to infinity with $t$. Hence, the weighted distance becomes small, and we conclude by a union bound over \eqref{eq:innercore-bound-1} and \eqref{eq:innercore-bound-2} twice (for both $w_1$ and $w_2$) that the result \eqref{eq:bounded_distance_inner} follows for $t$ large.
  \qed

  % \bibliographystyle{abbrv}
  % \bibliography{pa_weighted_scalefree}
\end{document}